\documentclass{amsart}
\usepackage{graphicx}
\usepackage{amsmath}
\usepackage{amscd}
\usepackage{mathtools}
\usepackage{enumerate}
\usepackage{tikz-cd}
\usepackage[all,cmtip]{xy}
\usepackage{tikz}

\newtheorem{dfn}{Definition}[section]
\newtheorem{thm}[dfn]{Theorem}
\newtheorem{pro}[dfn]{Proposition}
\newtheorem{lem}[dfn]{Lemma}
\newtheorem{cro}[dfn]{Corollary}
\newtheorem{mar}[dfn]{Remark}

\title[universal covering calabi-yau manifolds of $E^{[n]}$]{universal covering calabi-yau manifolds of the hilbert schemes of n points of enriques surfaces}
\author{Taro Hayashi}
\address{
(Taro Hayashi)
Department of Mathematics,
Graduate School of Science,
Osaka University,
Machikaneyamacho 1-1, Toyonaka, Osaka 560-0043, Japan
}
\email{tarou-hayashi@cr.math.sci.osaka-u.ac.jp}
\date{\today}
\begin{document}

\maketitle
\begin{abstract}
Throughout this paper, we work over ${\mathbb C}$, and $n$ is an integer such that $n\geq 2$. 
For an Enriques surface $E$, let $E^{[n]}$ be the Hilbert scheme of $n$ points of $E$. 
By Oguiso and Schr\"oer $[\ref{bio:3},\,{\rm Theorem}\,3.1]$, $E^{[n]}$ has a Calabi-Yau manifold $X$ as the universal covering space, $\pi :X\rightarrow E^{[n]}$ of degree $2$.
The purpose of this paper is to investigate a relationship of the small deformation of $E^{[n]}$ and that of $X$ $({\rm Theorem}\ 1.1)$, the natural automorphism of $E^{[n]}$ $({\rm Theorem}\,1.2)$,
and count the number of isomorphism classes of the Hilbert schemes of $n$ points of Enriques surfaces which has $X$ as the universal covering space when we fix one $X$ $({\rm Theorem}\,1.3)$.
\end{abstract}

\section{Introduction}
Throughout this paper, we work over ${\mathbb C}$, and $n$ is an integer such that $n\geq 2$. 
For an Enriques surface $E$, let $E^{[n]}$ be the Hilbert scheme of $n$ points of $E$. 
By Oguiso and Schr\"oer $[\ref{bio:3},\,{\rm Theorem}\,3.1]$, $E^{[n]}$ has a Calabi-Yau manifold $X$ as the universal covering space, $\pi :X\rightarrow E^{[n]}$ of degree $2$.
The purpose of this paper is to investigate a relationship of the small deformation of $E^{[n]}$ and that of $X$ $({\rm Theorem}\,1.1)$, the natural automorphism of $E^{[n]}$ $({\rm Theorem}\,1.2)$,
and count the number of isomorphism classes of the Hilbert schemes of $n$ points of Enriques surfaces which has $X$ as the universal covering space when we fix one $X$ $({\rm Theorem}\,1.3)$.
\\

Small deformations of a smooth compact surface $S$ induce that of the Hilbert scheme of $n$ points of $S$ by taking the relative Hilbert scheme.
Let $K$ be a $K3$ surface.
By Beauville $[\ref{bio:6},\,{\rm page}\,779$-$781]$, a very general small deformation of $K^{[n]}$ is not isomorphic to the Hilbert scheme of $n$ points of a $K3$ surface.
On the other hand, by Fantechi $[\ref{bio:1},\,{\rm Theorems}0.1\,{\rm and}\,0.3]$, every small deformations of $E^{[n]}$ is induced by that of $E$.
Since $X$ is the universal covering of $E^{[n]}$, the small deformation of $E^{[n]}$ induces that of $X$.
We consider a relationship of the small deformation of $E^{[n]}$ and that of $X$.
Our first main result is following:
\begin{thm}\label{thm:1} 
Let $E$ be an Enriques surface, $E^{[n]}$ the Hilbert scheme of $n$ points of $E$, and $X$ the universal covering space of $E^{[n]}$. 
Then every small deformation of $X$ is induced by that of $E^{[n]}$. 
\end{thm}
Compare with the fact that a general small deformation of the universal covering $K3$ surface of $E$ is not induced by that of $E$.
\\

Next, we study the natural automorphisms of $E^{[n]}$.
Any automorphism $f\in $ Aut$(S)$ induces an automorphism $f^{[n]}\in $ Aut$(S^{[n]})$.
An automorphism $g\in $ Aut$(S^{[n]})$ is called natural if there is an automorphism $f\in $ Aut$(S)$ such that $g=f^{[n]}$.
When $K$ is a $K3$ surface,
the natural automorphisms of $K^{[n]}$ have been studied by Boissi\`ere and Sarti $[\ref{bio:5},\,{\rm Theorem}\,1]$. 
They used the global Torelli theorem for
$K3$ surfaces: an effective Hodge isometry $\alpha $ is induced by a unique automorphism $\beta $
of $K3$ surface such that $\alpha =\beta ^{\ast }$.
Our second main result is the following theorem, similar to $[\ref{bio:5},\,{\rm Theorem}\,1]$ without 
the Torelli theorem for Enriques surfaces by using a result of Oguiso $[\ref{bio:4},\,{\rm Proposition}\,4,4]$.
\begin{thm}\label{thm:2}
Let $E$ be an Enriques surface, $D_{E}$ the exceptional divisor of the Hilbert-Chow morphism $\pi _{E}:E^{[n]}\rightarrow E^{(n)}$, and $n\geq 2$. An automorphism $f$ of $E^{[n]}$ is natural if and only if $f(D_{E})=D_{E}$, i.e.\,$f^{\ast }({\mathcal O}_{E^{[n]}}(D_{E}))={\mathcal O}_{E^{[n]}}(D_{E})$. 
\end{thm}
Finally, we compute the number of isomorphism class of the Hilbert schemes of $n$ points of Enriques surfaces which have $X$ as the universal covering space when we fixed one $X$.
\begin{thm}\label{thm:3}
Let $E$ and $E'$ be two Enriques surfaces, $E^{[n]}$ and $E'^{[n]}$ the Hilbert scheme of $n$ points of $E$ and $E'$, $X$ and $X'$ the universal covering space of $E^{[n]}$ and $E'^{[n]}$, and $n\geq 3$.
If $X\cong X'$, then $E^{[n]}\cong E'^{[n]}$, i.e. when we fix $X$, then there is just one isomorphism class of the Hilbert schemes of $n$ points of Enriques surfaces such that they have it as the universal covering space.
\end{thm}
Our proof is based on Theorem $1.2$ and the study of the action of the covering involutions on $H^{2}(X,{\mathbb C})$.

This is the result that is greatly different from the result of Ohashi \\
$($See $[\ref{bio:10},\,{\rm Theorem}\,0.1])$ that, for any nonnegative integer $l$, there exists a $K3$ surface with exactly $2^{l+10}$ distinct Enriques quotients. In particular, there does not exist a universal bound for the number of distinct Enriques quotients of a $K3$ surface.
Here we will call two Enriques quotients of a K3 surface distinct if they are not isomorphic to each other.
\begin{mar}\label{dfn:1000}
{\rm When} n=2, {\rm I\ do\ not\ count\ the\ number\ of\ isomorphism\ classes\ of\ the Hilbert\ schemes\ of} n {\rm points\ of\ Enriques\ surfaces\ which\ has} X {\rm as\ the\ universal\ covering\ space\ when\ we\ fix\ one} X.
\end{mar}

\section{Preliminaries}
A $K3$ surface $K$ is a compact complex surface with $K_{K} \sim 0$ and $H^{1}(K,
{\mathcal O}_{K})=0$. An Enriques surface $E$ is a compact complex surface with
 $H^{1}(E, {\mathcal O}_{E} )=0$,
$H^{2}(E, {\mathcal O}_{E})=0$, $K_{E}\not\sim 0$, and $2K_{E}\sim 0$. The universal covering of an Enriques surface is
a $K3$ surface.
A Calabi-Yau manifold $X$ is an $n$-dimensional compact k\"ahler manifold such that it is simply connected, there is no holomorphic $k$-form on $X$
for $ 0 < k < n$ and there is a nowhere vanishing holomorphic $n$-form on $X$.

Let $S$ be a nonsingular surface,
$S^{[n]}$ the Hilbert scheme of $n$ points of $S$, $\pi _{S}:S^{[n]}\rightarrow S^{(n)}$ the Hilbert-Chow morphism, and
$p_{S}:S^{n}\rightarrow S^{(n)}$ the natural projection. We denote by $D_{S}$ the exceptional divisor of $\pi _{S}$.
Note that $S^{[n]}$ is smooth of dim$_{{\mathbb C}}S^{[n]}=2n$.
Let $\Delta _{S}^{n}$ be the set of $n$-uples $(x_{1},\ldots ,x_{n})\in S^{n}$ with at least two $x_{i}$'s equal,
$S^{n}_{\ast }$ the set of $n$-uples $(x_{1},\ldots ,x_{n})\in S^{n}$ with at most two $x_{i}$'s equal.
We put 
\[ S^{(n)}_{\ast }:=p_{S}(S^{n}_{\ast }), \] 
\[ \Delta _{S}^{(n)}:=p_{S}(\Delta _{S}^{n}), \]
\[ S^{[n]}_{\ast }:=\pi _{S}^{-1}(S^{(n)}_{\ast }), \]
\[ \Delta _{S\ast }^{n}:=\Delta _{S}^{n}\cap S^{n}_{\ast }, \]
\[ \Delta _{S\ast }^{(n)}:=p_{S}(\Delta _{S\ast }^{n}),\ {\rm and}\]
\[ F_{S}:=S^{[n]}\setminus S^{[n]}_{\ast }. \]
Then we have ${\rm Blow}_{\Delta _{S\ast }^{n}}S^{n}_{\ast }/{\mathcal S}_{n}\simeq S^{[n]}_{\ast }$,
$F_{S}$ is an analytic closed subset, and its codimension is $2$ in $S^{[n]}$ by Beauville $[\ref{bio:6},\,{\rm page}\,767$-$768]$.
Here ${\mathcal S}_{n}$ is the symmetric group of degree $n$ which acts naturally on $S^{n}$ by permuting of the factors.

Let $E$ be an Enriques surface, and $E^{[n]}$ the Hilbert scheme of $n$ points of $E$.
By Oguiso and Schr\"oer $[\ref{bio:3},\,{\rm Theorem}\,3.1]$, $E^{[n]}$ has a Calabi-Yau manifold $X$ as the universal covering space $\pi :X\rightarrow E^{[n]}$ of degree $2$.
Let $\mu :K\rightarrow E$ be the universal covering space of $E$ where $K$ is a $K3$ surface,
$S_{K}$ the pullback of $\Delta _{E}^{(n)}$ by the morphism 
\[ \mu ^{(n)}:K^{(n)}\ni [(x_{1},\ldots ,x_{n})]\mapsto [(\mu (x_{1}),\ldots ,\mu (x_{n}))]\in E^{(n)}. \]
Then we get a $2^{n}$-sheeted unramified covering space 
\[ \mu ^{(n)}|_{K^{(n)}\backslash S_{K} }:K^{(n)}\backslash S_{K} \rightarrow E^{(n)}\backslash \Delta _{E}^{(n)}.\]
Furthermore, let $\Gamma _{K}$ be the pullback of $S_{K}$ by natural projection $p_{K}:K^{n}\rightarrow K^{(n)}$.
Since $\Gamma _{K}$ is an algebraic closed set with codimension $2$,
then 
\[ \mu ^{(n)}\circ p_{K}:K^{n}\backslash \Gamma _{K}\rightarrow E^{(n)}\backslash \Delta _{E}^{(n)} \]
is the $2^{n}n!$-sheeted universal covering space. 
Since $E^{[n]}\backslash D_{E}=E^{(n)}\backslash \Delta _{E}^{(n)}$ where $D_{E}=\pi _{E}^{-1}(\Delta _{E}^{(n)})$, we regard the universal covering space $\mu ^{(n)}\circ p_{K}:K^{n}\backslash \Gamma _{K} \rightarrow E^{(n)}\backslash \Delta _{E}^{(n)}$ 
as the universal covering space of $E^{[n]}\setminus D_{E}$:
\[ \mu ^{(n)}\circ p_{K}:K^{n}\backslash \Gamma _{K} \rightarrow E^{[n]}\backslash D_{E}. \]
Since $\pi :X\setminus \pi ^{-1}(D_{E})\rightarrow E^{[n]}\setminus D_{E}$ is a covering space and $\mu ^{(n)}\circ p_{K}:K^{n}\setminus \Gamma _{K}\rightarrow E^{[n]}\setminus D_{E}$ is the universal covering space,
there is a morphism 
\[ \omega :K^{n}\setminus \Gamma _{K}\rightarrow X\setminus \pi ^{-1}(D_{E}) \]
such that $\omega :K^{n}\setminus \Gamma _{K}\rightarrow X\setminus \pi ^{-1}(D_{E})$ is the universal covering space and $\mu ^{(n)}\circ p_{K}=\pi \circ \omega $:
$$
\xymatrix{
K^{n}\setminus \Gamma _{K} \ar[dr]_{\mu ^{(n)}\circ p_{K}} \ar[r]^{\omega } &X\setminus \pi ^{-1}(D_{E}) \ar[d]^{\pi } \\
&E^{[n]}\setminus D_{E}.
}
$$
We denote the covering transformation group of $\pi \circ \omega $
by:  
\[ G:=\{g\in {\rm Aut}(K^{n}\setminus \Gamma _{K}):\pi \circ \omega \circ g=\pi \circ \omega \}. \] 
Then $G$ is of order $2^{n}.n!$, since deg$(\mu ^{(n)}\circ p_{K})=2^{n}.n!$. 
Let $\sigma $ be the covering involution of $\mu :K\rightarrow E$, and for
\[ 1\leq k\leq n,\ 1\leq i_{1}<\cdots <i_{k}\leq n \] 
we define automorphisms $\sigma _{i_{1}\ldots i_{k}}$ of $K^{n}$ 
by following.
For $x=(x_{i})_{i=1}^{n}\in K^{n}$,
$$
the\ j\mathchar`-th\ component\ of\ \sigma _{i_{1}\ldots i_{k}}(x)=
\begin{cases} 
  \sigma (x_{j})  & j\in \{i_{1},\cdots ,i_{k}\} \\
  x_{j} &  j\not\in \{i_{1},\cdots ,i_{k}\}.
\end{cases}
$$
Then ${\mathcal S}_{n}\subset G$, and $\{\sigma _{i_{1}\ldots i_{k}}\}_{1\leq k\leq n,\ 1\leq i_{1}<\ldots <i_{k}\leq n}\subset G$. 
Let $H$ be the subgroup of $G$ generated by ${\mathcal S}_{n}$ and $\{\sigma _{ij}\}_{1\leq i<j\leq n}.$
\begin{pro}\label{dfn:3}
$G$ is generated by ${\mathcal S}_{n}$ and $\{\sigma _{i_{1}\ldots i_{k}}\}_{1\leq k\leq n,\ 1\leq i_{1}<\ldots <i_{k}\leq n}$. Moreover any element is of the form $s\circ t$ where $s\in {\mathcal S}_{n}$, $t\in \{\sigma _{i_{1}\ldots i_{k}}\}_{1\leq k\leq n,\ 1\leq i_{1}<\ldots <i_{k}\leq n}$.
\end{pro}
\begin{proof}
If $(s,t)=(s',t')$ for $s,s'\in {\mathcal S}_{n}$ and $t,t'\in \{\sigma _{i_{1}\ldots i_{k}}\}_{1\leq k\leq n,\ 1\leq i_{1}<\ldots <i_{k}\leq n}$, then we have $s=s'$ and $t=t'$ by paying attention to the permutation of component.
As $|{\mathcal S}_{n}|=n!$, and $|\{\sigma _{i_{1}\ldots i_{k}}\}_{1\leq k\leq n,\ 1\leq i_{1}<\ldots <i_{k}\leq n}|=2^{n}$,
$G$ is generated by ${\mathcal S}_{n}$ and $\{\sigma _{i_{1}\ldots i_{k}}\}_{1\leq k\leq n,\ 1\leq i_{1}<\ldots <i_{k}\leq n}$.
\end{proof}
\begin{pro}\label{dfn:4}
$|H|=2^{n-1}.n!$.
\end{pro}
\begin{proof}
$H$ is generated by ${\mathcal S}_{n}$ and $\{\sigma _{ij}\}_{1\leq i<j\leq n}$.  By paying attention to the permutation of component, we have $\sigma _{i}\not\in H$ for all $i$.
For arbitrary $j$, $(i,j)\circ \sigma _{i}\circ (i,j)=\sigma _{j}$. Since ${\mathcal S}_{n}\subset H$, and  Proposition $2.1$, we obtain $|G/H|=2$, i.e. $|H|=2^{n-1}.n!$.
\end{proof}

We put
\[ K^{n}_{\ast \mu }:=({\mu ^{n}})^{-1}(E^{n}_{\ast }), \] 
where $\mu ^{n}:K^{n}\ni (x_{i})_{i=1}^{n}\mapsto (\mu (x_{i}))_{i=1}^{n}\in E^{n}$. Recall that $\mu :K\rightarrow E$ the universal covering with $\sigma $ the covering involution.
We further put 
\[ T_{\ast \mu \,ij}:=\{ (x_{l})_{l=1}^{n}\in K^{n}_{\ast \mu }: \sigma (x_{i})=x_{j} \}, \]
\[ \Delta _{K\ast \mu \,ij}:=\{ (x_{l})_{l=1}^{n}\in K^{n}_{\ast \mu }: x_{i}=x_{j} \}, \]
\[ T_{\ast \mu }:=\bigcup _{1\leq i<j\leq n}T_{\ast \mu \,i,j},\ {\rm and}\]
\[ \Delta _{K\ast \mu }:=\bigcup _{1\leq i<j\leq n}\Delta _{K\ast \mu \,i,j}. \]
By the definition of $K^{n}_{\ast \mu }$, $H$ acts on $K^{n}_{\ast \mu }$, and by the definition of $\Delta _{K\ast \mu }$ and $T_{\ast \mu }$, we have $\Delta _{K\ast \mu }\cap T_{\ast \mu }=\emptyset $.
\begin{lem}\label{dfn:20}
For $t\in H$ and $1\leq i<j\leq n$, if $t\in H$ has a fixed point on $\Delta _{K\ast \mu \,ij}$, then  $t=(i,j)$ or $t={\rm id}_{K^{n}}$.
\end{lem}
\begin{proof}
Let $t\in H$ be an element of $H$ where there is an element $\tilde{x}=(\tilde{x}_{i})_{i=1}^{n}\in \Delta _{K\ast \mu \,ij}$ such that $t(\tilde{x})=\tilde{x}$.
By Proposition $\ref{dfn:3}$, for $t\in H$, there are two elements $\sigma _{i_{1},\cdots ,i_{k}}\in \{\sigma _{i_{1}\ldots i_{k}}\}_{1\leq k\leq n,\ 1\leq i_{1}<\ldots <i_{k}\leq n}$ and $(j_{1},\cdots ,j_{l})\in {\mathcal S}_{n}$ such that 
\[ t=(j_{1},\cdots ,j_{l})\circ \sigma _{i_{1},\cdots ,i_{k}}. \]
From the definition of $\Delta _{K\ast \mu \,ij}$, for $(x_{l})_{l=1}^{n}\in \Delta _{K\ast \mu \,ij}$, 
\[ \{ x_{1},\ldots ,x_{n} \}\cap \{\sigma (x_{1}),\ldots ,\sigma (x_{n}) \}=\emptyset. \] 
Suppose $\sigma _{i_{1},\cdots ,i_{k}}\not={\rm id}_{K^{n}}$. Since $t(\tilde{x})=\tilde{x}$,
we have 
\[ \{ \tilde{x}_{1},\ldots ,\tilde{x}_{n} \}\cap \{\sigma (\tilde{x}_{1}),\ldots ,\sigma (\tilde{x}_{n}) \}\not=\emptyset, \] 
a contradiction. Thus we have $t=(j_{1},\cdots ,j_{l})$. Similarly from the definition of $\Delta _{K\ast \mu \,ij}$,
for $(x_{l})_{l=1}^{n}\in \Delta _{K\ast \mu \,ij}$,
if $x_{s}=x_{t}$ $(1\leq s<t\leq n)$, then $s=i$ and $t=j$. 
Thus we have $t=(i,j)$ or $t={\rm id}_{K^{n}}$. 
\end{proof}
\begin{lem}\label{dfn:21}
For $t\in H$ and $1\leq i<j\leq n$, if $t\in H$ has a fixed point on $T_{\ast \mu \,ij}$, then  $t=\sigma _{i,j}\circ (i,j)$ or $t={\rm id}_{K^{n}}$.
\end{lem}
\begin{proof}
Let $t\in H$ be an element of $H$ where there is an element $\tilde{x}=(\tilde{x}_{i})_{i=1}^{n}\in T_{K\ast \mu \,ij}$ such that $t(\tilde{x})=\tilde{x}$.
By Proposition $\ref{dfn:3}$, for $t\in H$, there are two elements $\sigma _{i_{1},\cdots ,i_{k}}\in \{\sigma _{i_{1}\ldots i_{k}}\}_{1\leq k\leq n,\ 1\leq i_{1}<\ldots <i_{k}\leq n}$ and $(j_{1},\cdots ,j_{l})\in {\mathcal S}_{n}$ such that
\[ t=(j_{1},\cdots ,j_{l})\circ \sigma _{i_{1},\cdots ,i_{k}}. \]
Since $(j,j+1)\circ \sigma _{i,j}\circ (j,j+1):\Delta _{K\ast \mu \,ij}\rightarrow T_{\ast \mu \,ij}$ is an isomorphism,
and by Lemma $\ref{dfn:20}$, we have 
\[ (j,j+1)\circ \sigma _{i,j}\circ (j,j+1)\circ t\circ (j,j+1)\circ \sigma _{i,j}\circ (j,j+1)=(i,j)\ {\rm or}\ {\rm id}_{K^{n}}. \]
If $(j,j+1)\circ \sigma _{i,j}\circ (j,j+1)\circ t\circ (j,j+1)\circ \sigma _{i,j}\circ (j,j+1)={\rm id}_{K^{n}}$, then $t={\rm id}_{K^{n}}$.
If $(j,j+1)\circ \sigma _{i,j}\circ (j,j+1)\circ t\circ (j,j+1)\circ \sigma _{i,j}\circ (j,j+1)=(i,j)$, then
\begin{equation*}
\begin{split}
t&=(j,j+1)\circ \sigma _{i,j}\circ (j,j+1)\circ (i,j)\circ (j,j+1)\circ \sigma _{i,j}\circ (j,j+1)\\
&=(j,j+1)\circ \sigma _{i,j}\circ (i,j+1)\circ \sigma _{i,j}\circ (j,j+1)\\
&=(j,j+1)\circ \sigma _{i,j+1}\circ (i,j+1)\circ (j,j+1)\\
&=\sigma _{i,j}\circ (i,j).
\end{split}
\end{equation*}
Thus we have $t=\sigma _{i,j}\circ (i,j)$.
\end{proof}
From Lemma $\ref{dfn:20}$ and Lemma $\ref{dfn:21}$, the universal covering map $\mu $ induces a local isomorphism 
\[ \mu ^{[n]}_{\ast }:{\rm Blow}_{\Delta _{K\ast \mu }\cup T_{\ast \mu }}K^{n}_{\ast \mu }/H\rightarrow {\rm Blow}_{\Delta _{E\ast }^{n}}E^{n}_{\ast }/{\mathcal S}_{n}=E^{[n]}_{\ast }. \]
Here Blow$_{A}B$ is the blow up of $B$ along $A\subset B$.
\begin{lem}\label{dfn:22}
For every $x\in E^{[n]}_{\ast }$, $|({\mu ^{[n]}_{\ast }})^{-1}(x)|=2$.
\end{lem}
\begin{proof}
For $(x_{i})_{i=1}^{n}\in \Delta _{E\ast }^{n}$ with $x_{1}=x_{2}$, there are $n$ elements $y_{1},\ldots ,y_{n}$ of $K$ such that $y_{1}=y_{2}$ and $\mu (y_{i})=x_{i}$ for $1\leq i\leq n$.
Then 
\[ ({\mu ^{n}})^{-1}((x_{i})_{i=1}^{n})\cap K^{n}_{\ast \mu }=\{y_{1},\sigma (y_{1})\}\times \cdots \times \{y_{n},\sigma (y_{n})\}. \]
For $\sigma _{i_{1}\ldots i_{k}}\in G$, since $H$ is generated by ${\mathcal S}_{n}$ and $\sigma _{i_{1}\ldots i_{k}}$, if $k$ is even we get $\sigma _{i_{1}\ldots i_{k}}\in H$, if $k$ is odd $\sigma _{i_{1}\ldots i_{k}}\not\in H$.
For $\{z_{i}\}_{i=1}^{n}\in ({\mu ^{n}})^{-1}((x_{i})_{i=1}^{n})\cap K^{n}_{\ast \mu }$, if the number of $i$ with $z_{i}=y_{i}$ is even then 
\[ \{z_{i}\}_{i=1}^{n}=\{\sigma (y_{1}),\sigma (y_{2}),y_{3}\ldots ,y_{n}\}\ {\rm on}\ K^{n}_{\ast \mu }/H,\ {\rm and} \]
if the number of $i$ with $z_{i}=y_{i}$ is odd then 
\[ \{z_{i}\}_{i=1}^{n}=\{\sigma (y_{1}),y_{2},y_{3}\ldots ,y_{n}\}\ {\rm on}\ K^{n}_{\ast \mu }/H. \]
Furthermore since $\sigma _{i}\not\in H$ for $1\leq i\leq n$,
\[ \{\sigma (y_{1}),\sigma (y_{2}),y_{3}\ldots ,y_{n}\}\not=\{\sigma (y_{1}),y_{2},y_{3}\ldots ,y_{n}\}. \]
Thus for every $x\in E^{[n]}_{\ast }$, $|({\mu ^{[n]}_{\ast }})^{-1}(x)|=2$.
\end{proof}
\begin{pro}\label{dfn:30}
$\mu ^{[n]}_{\ast }:{\rm Blow}_{\Delta _{K\ast \mu }\cup T_{\ast \mu }}K^{n}_{\ast \mu }/H\rightarrow {\rm Blow}_{\Delta _{E\ast }^{n}}E^{n}_{\ast }/{\mathcal S}_{n}$ is the universal covering space, and $X\setminus \pi ^{-1}(F_{E})\simeq {\rm Blow}_{\Delta _{K\ast \mu }\cup T_{\ast \mu }}K^{n}_{\ast \mu }/H$.
\end{pro}
\begin{proof}
Since $\mu ^{[n]}_{\ast }$ is a local isomorphism and the number of fiber is constant, so $\mu ^{[n]}_{\ast }$ is a covering map. Furthermore $\pi :X\setminus \pi ^{-1}(F_{E})\rightarrow E^{[n]}_{\ast }$ is the universal covering space and number of fiber is $2$, so  $\mu ^{[n]}_{\ast }:{\rm Blow}_{\Delta _{K\ast \mu }\cup T_{\ast \mu }}K^{n}_{\ast \mu }/H\rightarrow {\rm Blow}_{\Delta _{E\ast }^{n}}E^{n}_{\ast }/{\mathcal S}_{n}$ is the universal covering space, and by the uniqueness of the universal covering space, we have $X\setminus \pi ^{-1}(F_{E})\simeq {\rm Blow}_{\Delta _{K\ast \mu }\cup T_{\ast \mu }}K^{n}_{\ast \mu }/H$.
\end{proof}
Recall that $H$ is generated by ${\mathcal S}_{n}$ and $\{\sigma _{ij}\}_{1\leq i<j\leq n}$.
\begin{thm}\label{thm:10}
Let $E$ be an Enriques surface, $E^{[n]}$ the Hilbert scheme of $n$ points of $E$, $\pi :X\rightarrow E^{[n]}$ the universal covering space of $E^{[n]}$, and $n\geq 2$.
Then there is a resolution $\varphi  _{X}:X\rightarrow K^{n}/H$ such that $\varphi _{X}^{-1}(\Gamma _{K}/H)=\pi ^{-1}(D_{E})$.
\end{thm}
\begin{proof}
Let $E$ be an Enriques surface, $E^{[n]}$ the Hilbert scheme of $n$ points of $E$, $\pi :X\rightarrow E^{[n]}$ the universal covering space of $E^{[n]}$ where $X$ is a Calabi-Yau manifold, and $\rho $ the covering involution of $\pi $.
From ${\rm Proposition}2.6$, we have $X\setminus \pi ^{-1}(F_{E})\simeq {\rm Blow}_{\Delta _{K\ast \mu }\cup T_{\ast \mu }}K^{n}_{\ast \mu }/H$. Thus there is a meromorphim $f$ of $X$ to $K^{n}/H$ with satisfying the following  commutative diagram:
$$
\xymatrix{
E^{[n]}\setminus F_{E} \ar[r]^{\pi _{E}}&E^{(n)} \\
X\setminus \pi ^{-1}(F_{E}) \ar[u]^{\pi } \ar[r]^{f} &K^{n}/H \ar[u]^{p_{H}}
}
$$
where $\pi _{E}:E^{[n]}\rightarrow E^{(n)}$ is the Hilbert-Chow morphism, and $p_{H}:K^{n}/H\rightarrow E^{(n)}$ is the natural projection.
For any ample line bundle ${\mathcal L}$ on $E^{(n)}$, since the natural projection $p_{H}:K^{n}/H\rightarrow E^{(n)}$ is finite, and $E^{(n)}$ and $K^{n}/H$ are projective, $p_{H}^{\ast }{\mathcal L}$ is ample.
Since $\pi ^{-1}(F_{E})$ is an analytic closed subset of codimension $2$ in $X$, there is a line bundle ${\mathbb L}$ on $X$ such that 
$f^{\ast }(p_{H}^{\ast }{\mathcal L})={\mathbb L}\mid _{X\setminus \pi ^{-1}(F_{E})}$.
From the above diagram, we have
\[ {\mathbb L}=\pi ^{\ast }(\pi _{E}^{\ast }{\mathcal L}).\]
Since ${\mathcal L}$ is ample on $E^{(n)}$, $\pi _{E}^{\ast }{\mathcal L}$ is a globally generated line bundle on $E^{[n]}$. 
Moreover $\pi ^{\ast }(\pi _{E}^{\ast }{\mathcal L})$ is also a globally generated line bundle on $X$.
Since $p_{H}^{\ast }{\mathcal L}$ is ample on $K^{n}/H$ and ${\mathbb L}$ is globally generated, 
there is a holomorphism $\varphi _{X}$ of $X$ to $K^{n}/H$ such that $\varphi _{X}\mid _{X\setminus \pi ^{-1}(F_{E})}=f\mid _{X\setminus \pi ^{-1}(F_{E})}$. 
Since $X$ is a proper and the image of $f$ contains a Zariski open subset, $\varphi _{X}:X\rightarrow K^{n}/H$ is surjective.
Moreover $f:X\setminus \pi ^{-1}(D_{E})\cong (K^{n}\setminus \Gamma _{K})/H$, that is a resolution.
\end{proof}

\section{Proof of Theorem $1.1$}
Let $S$ be a smooth projective surface and $P(n)$ the set of partitions of $n$. 
We write $\alpha \in P(n)$ as $\alpha =(\alpha _{1},\ldots ,\alpha _{n})$ with $1\cdot {\alpha _{1}}+\cdots+n\cdot {\alpha _{n}}=n$, 
and put $|\alpha |:= \sum _{i}\alpha _{i}$.
We put $S^{\alpha }:=S^{\alpha _{1}}\times \cdots \times S^{\alpha _{n}}$, $S^{(\alpha )}:=S^{(\alpha _{1})}\times \cdots \times S^{(\alpha _{n})}$ and $S^{[n]}$ the Hilbert scheme of $n$ points of $S$.
The cycle type $\alpha (g)$ of $g\in {\mathcal S}_{n}$ is the partition $(1^{\alpha _{1}(g)},\ldots  ,n^{\alpha _{n}(g)})$ where $\alpha _{i}(g)$ is the number of cycles with length $i$ as the 
representation of $g$ in a product of disjoint cycles. As usual, we denote by $(n_{1},\ldots ,n_{r})$ the cycle defined by mapping $n_{i}$ to $n_{i+1}$ for $i < r$ and $n_{r}$ to $n_{1}$.
By Steenbrink $[\ref{bio:21},\,{\rm page}\,526$-$530]$, $S^{(\alpha )}\,(\alpha \in P(n))$ have the Hodge decomposition.
By G\"ottsche and Soergel $[\ref{bio:2},\,{\rm Theorem}\,2]$, we have an isomorphism of Hodge structures:
\[ H^{i+2n}(S^{[n]},{\mathbb C})(n) = \sum _{\alpha \in P(n)}H^{i+2|\alpha |}(S^{(\alpha )},{\mathbb C})(|\alpha |) \]
where $H^{i+2|\alpha |}(S^{(\alpha )},{\mathbb C})(|\alpha |)$ is the Tate twist of $H^{i+2|\alpha |}(S^{(\alpha )},{\mathbb C})$, \\
and $H^{i+2n}(S^{[n]},{\mathbb C})(n)$ is the Tate twist of $H^{i+2n}(S^{[n]},{\mathbb C})$.
Since $H^{i+2n}(S^{[n]},{\mathbb C})(n)$ is a Hodge structure of weight $i+2n-2n=i$, 
we have $H^{i+2n}(S^{[n]},{\mathbb C})(n)^{p,q}=H^{i+2n}(S^{[n]},{\mathbb Q})^{p+n,q+n}$ for $p,q\in {\mathbb Z}$ with $p+q=i$,
and $H^{i+2|\alpha |}(S^{(\alpha )},{\mathbb C})(|\alpha |)$ is a Hodge structure of weight $i+2|\alpha |-2|\alpha |=i$,
we have $H^{i+2|\alpha |}(S^{(\alpha )},{\mathbb C})(|\alpha |)^{p,q}=H^{i+2|\alpha |}(S^{(\alpha )},{\mathbb C})^{p+|\alpha |,q+|\alpha |}$ for $p,q\in {\mathbb Z}$ with $p+q=i$.
Thus we have 
\begin{equation}
\label{eq:1012}
{\rm dim}_{{\mathbb C}}H^{2n}(S^{[n]},{\mathbb C})^{1,2n-1} = \sum _{\alpha \in P(n)}{\rm dim}_{{\mathbb C}}H^{2|\alpha |}(S^{(\alpha )},{\mathbb C})^{1-n+|\alpha |,n-1+|\alpha |}.
\end{equation}
Let $E$ be an Enriques surface, $E^{[n]}$ the Hilbert scheme of $n$ points of $E$, and $\pi :X\rightarrow E^{[n]}$ the universal covering space of $E^{[n]}$ where $X$ is a Calabi-Yau manifold.
\begin{pro}\label{dfn:9}
${\rm dim}_{{\mathbb C}}H^{1}(E^{[n]}, \Omega ^{2n-1}_{E^{[n]}})=0$.
\end{pro}
\begin{proof}
From $[\ref{bio:21},\,{\rm page}\,526$-$530]$, $E^{(n)}$ have the Hodge decomposition, we have 
\[ H^{2n}(E^{[n]},{\mathbb C})^{1,2n-1}\simeq H^{2n-1}(E^{[n]}, \Omega ^{1}_{E^{[n]}}),\ {\rm and} \]
\[ H^{2n}(E^{(n)},{\mathbb C})^{1,2n-1}\simeq H^{2n-1}(E^{(n)}, \Omega ^{1}_{E^{(n)}}).\]
Similarly since $E^{(\alpha )}$ $(\alpha \in P(n))$ has the Hodge decomposition, 
if $1-n+|\alpha |<0$ or $n-1+|\alpha |>2n$ for $\alpha \in P(n)$, then 
\[ H^{2|\alpha |}(E^{(\alpha )},{\mathbb C})(|\alpha |)^{1-n+|\alpha |,n-1+|\alpha |}=0. \]
For $\alpha \in P(n)$ with $1-n+|\alpha |\geq 0$ and $n-1+|\alpha |\leq 2n$, then $|\alpha |=n-1$, $|\alpha |=n$ or $|\alpha |=n+1$. 
By the definition of $\alpha \in P(n)$ and $|\alpha |$,
we obtain $\alpha =\{(n,0,\ldots ,0),(n-2,1,0,\ldots ,0) \}$.
Thus, by the above equation $(\ref{eq:1012})$, we have 
\[ {\rm dim}_{{\mathbb C}}H^{2n}(E^{[n]}, {\mathbb C})^{1,2n-1}={\rm dim}_{{\mathbb C}}H^{2n}(E^{(n)},{\mathbb C})^{1,2n-1}\oplus H^{2n-2}(E^{(n-2)}\times E^{(2)},{\mathbb C})^{0,2n-2}.
 \]
 From the K\"unneth Theorem, we obtain
 \[ H^{2n-2}(E^{(n-2)}\times E^{(2)},{\mathbb C})^{0,2n-2}\simeq \bigoplus _{s+t=2n-2}H^{s}(E^{(n-2)},{\mathbb C})^{0,s}\otimes H^{t}(E^{(2)},{\mathbb C})^{0,t}. \]
 Since $H^{1}(E,{\mathbb C})^{0,1}=H^{2}(E,{\mathbb C})^{0,2}=0$, we have 
\[ H^{2n-2}(E^{(n-2)}\times E^{(2)},{\mathbb C})^{0,2n-2}=0. \]
Let $\Lambda $ be a subset of ${\mathbb Z}_{\geq 0}^{2n}$
\[ \Lambda :=\{ (s_{1},\cdots ,s_{n},t_{1},\cdots ,t_{n})\in {\mathbb Z}_{\geq 0}^{2n}:\Sigma _{i=1}^{n}s_{i}=1,\,\Sigma _{j=1}^{n}t_{j}=2n-1 \}. \] 
From the K\"unneth Theorem, we have
$$
H^{2n}(E^{n},{\mathbb C})^{1,2n-1}\simeq \bigoplus _{(s_{1},\cdots ,s_{n},t_{1},\cdots ,t_{n})\in \Lambda }\biggl{(}\bigotimes _{i=1}^{n}H^{2}(E,{\mathbb C})^{s_{i},t_{i}}\biggl{)}.
$$ 
Since $n\geq 2$,
for each $(s_{1},\cdots ,s_{n},t_{1},\cdots ,t_{n})\in \Lambda $, there is a number $i\in \{1,\cdots ,n\}$ such that $s_{i}=0$.
Thus since $H^{2}(E,{\mathbb C})^{0,2}=0$, we have $H^{2n-1}(E^{n},{\mathbb C})^{1,2n-1}=0$, so $H^{2n-1}(E^{(n)},{\mathbb C})^{1,2n-1}=0$.
Hence $H^{1}(E^{[n]}, \Omega ^{2n-1}_{E^{[n]}})=0$.
\end{proof}
\begin{thm}\label{thm:1} 
Let $E$ be an Enriques surface, $E^{[n]}$ the Hilbert scheme of $n$ points of $E$, and $X$ the universal covering space of $E^{[n]}$. 
Then all small deformations of $X$ is induced by that of $E^{[n]}$. 
\end{thm}
\begin{proof} 
Since each canonical bundle of $E$ and $E^{[n]}$ is torsion, and from Ran $[\ref{bio:20},\,{\rm Corollary}\,2]$, they have unobstructed deformations.
The Kuranishi family of $E$ has a $10$-dimensional smooth base, so the Kuranishi family of $E^{[n]}$ has a 
$10$-dimensional smooth base 
by $[\ref{bio:1},\,{\rm Theorems}\,0.1\,{\rm and}\,0.3]$. Thus we have dim$_{{\mathbb C}}H^{1}(E^{{n}},T_{E^{[n]}})=10$. 

Since $K_{E^{[n]}}$ is not trivial and $2K_{E^{[n]}}$ is trivial, we have 
\[ T_{E^{[n]}}\simeq \Omega ^{2n-1}_{E^{[n]}}\otimes K_{E^{[n]}}. \]
Therefore we have dim$_{{\mathbb C}}H^{1}(E^{{n}},\Omega ^{2n-1}_{E^{[n]}}\otimes K_{E^{[n]}})=10$.
Since $K_{X}$ is trivial, then we have $T_{X}\simeq \Omega ^{2n-1}_{X}$. 
Since $\pi :X\rightarrow E^{[n]}$ is the covering map and 
\[ X\simeq {\mathcal Spec}\,{\mathcal O}_{E^{[n]}}\oplus {\mathcal O}_{E^{[n]}}(K_{E^{[n]}}) \]
 by $[\ref{bio:3},\,{\rm Theorem}\,3.1]$, we have
\begin{equation*}
\begin{split}
H^{k}(X,\Omega ^{2n-1}_{X})&\simeq H^{k}(E^{[n]},\Omega ^{2n-1}_{E^{[n]}}\oplus (\Omega ^{2n-1}_{E^{[n]}}\otimes K_{E^{[n]}}))\\
&\simeq H^{k}(E^{[n]},\Omega ^{2n-1}_{E^{[n]}})\oplus H^{k}(E^{[n]},\Omega ^{2n-1}_{E^{[n]}}\otimes K_{E^{[n]}}).
\end{split}
\end{equation*}
Combining this with Proposition $3.1$, we obtain
\[ {\rm dim}_{\mathbb C}H^{1}(X,\Omega ^{2n-1}_{X})={\rm dim}_{\mathbb C}H^{1}(E^{[n]},\Omega ^{2n-1}_{E^{[n]}}\otimes K_{E^{[n]}}). \]
Since $\pi :X\rightarrow E^{[n]}$ is a covering map, $\pi ^{\ast }:H^{1}(E^{[n]},T_{E^{[n]}})\hookrightarrow H^{1}(X,T_{X})$ is injective.
Thus we have dim$_{\mathbb C}H^{1}(X,T_{X})=10$.

Let $p:{\mathcal Y}\rightarrow U$ be the universal family of $E^{[n]}$ and $f:{\mathcal X}\rightarrow {\mathcal Y}$ is the universal covering space.
Then $q:{\mathcal X}\rightarrow U$ is a flat family of $X$ where $q:=p\circ f$.
Then we have a commutative diagram:
$$
\xymatrix{
T_{U,0} \ar[dr]_{\rho _{q}} \ar[r]^{\rho_{p}} &{\rm H}^{1}({\mathcal Y}_{0},T_{{\mathcal Y}_{0}}) \ar[d]^{\tau } \ar@{=}[r]  &H^{1}(E^{[n]},T_{E^{[n]}}) \ar[d]^{\pi ^{\ast }} \\
&{\rm H}^{1}({\mathcal X}_{0},T_{{\mathcal X}_{0}}) \ar@{=}[r] &H^{1}(X,T_{X}). 
}
$$
Since $H^{1}(E^{[n]},T_{E^{[n]}})\simeq H^{1}(X,T_{X})$ by $\pi ^{\ast }$, the vertical arrow $\tau $ is an isomorphism and 
\[ {\rm dim}_{{\mathbb C}}H^{1}({\mathcal X}_{u},T_{{\mathcal X}_{u}})={\rm dim}_{{\mathbb C}}H^{1}({\mathcal X}_{u},\Omega ^{2n-1}_{{\mathcal X}_{u}}) \]
is a constant for some neighborhood of $0\in U$,
it follows that $q:{\mathcal X}\rightarrow U$ is the complete family of ${\mathcal X}_{0}=X$, therefore $q:{\mathcal X}\rightarrow U$
is the versal family of ${\mathcal X}_{0}=X$. Thus every fibers of any small deformation of $X$ is the universal covering of some the Hilbert scheme of $n$ points of some Enriques surface.
\end{proof}

\section{Proof of Theorem $1.2$}
Let $E$ be an Enriques surface, $E^{[n]}$ the Hilbert scheme of $n$ points of $E$, and $\pi :X\rightarrow E^{[n]}$ the universal covering space of $E^{[n]}$ where $X$ is a Calabi-Yau manifold.
At first, we show that for an automorphism $f$ of $E^{[n]}$, $f(D_{E})=D_{E}\Leftrightarrow f^{\ast }({\mathcal O}_{E^{[n]}}(D_{E}))={\mathcal O}_{E^{[n]}}(D_{E})$. Next, we show Theorem $1.2$.
\begin{pro}\label{dfn:80}
${\rm dim}_{{\mathbb C}}H^{0}(E^{[n]},{\mathcal O}_{E^{[n]}}(D_{E}))=1$.
\end{pro}
\begin{proof}
Since $D_{E}$ is effective,
we obtain dim$_{{\mathbb C}}H^{0}(E^{[n]},{\mathcal O}_{E^{[n]}}(D_{E}))\geq 1$.
Since the codimension of $\Delta _{E}^{(n)}$ is $2$ in $E^{(n)}$, and $E^{(n)}$ is normal,
we have 
\[ H^{0}(E^{(n)},{\mathcal O}_{E^{(n)}})=\Gamma (E^{(n)}\setminus \Delta _{E}^{(n)},{\mathcal O}_{E^{(n)}}).\]
Since $\pi _{E}|_{E^{[n]}\backslash D_{E}}:E^{[n]}\backslash D_{E}\simeq E^{(n)}\backslash \Delta _{E}^{(n)}$, and ${\mathcal O}_{E^{[n]}}(D_{E})\simeq {\mathcal O}_{E^{[n]}}$ on $E^{[n]}\setminus D_{E}$,
we have
\[ (\pi _{E})_{\ast }({\mathcal O}_{E^{[n]}}(D_{E}))\simeq {\mathcal O}_{E^{(n)}}\ {\rm on}\ E^{(n)}\setminus \Delta _{E}^{(n)}. \]
Hence 
\[ \Gamma (E^{[n]}\setminus D_{E},{\mathcal O}_{E^{[n]}}(D_{E}))\simeq H^{0}(E^{(n)},{\mathcal O}_{E^{(n)}}). \]
Since $E^{(n)}$ is compact, we have $H^{0}(E^{(n)},{\mathcal O}_{E^{(n)}})\simeq {\mathbb C}$. Therefore we have
\[ {\rm dim}_{\mathbb C}\Gamma (E^{[n]}\setminus D_{E},{\mathcal O}_{E^{[n]}}(D_{E}))=1. \]
Thus we obtain ${\rm dim}_{{\mathbb C}}H^{0}(E^{[n]},{\mathcal O}_{E^{[n]}}(\pi ^{\ast }(D_{E})))=1$.
\end{proof}

\begin{mar}\label{dfn:1001}
{\rm Then\ by} ${\rm Proposition}\,\ref{dfn:80}$, {\rm for\ an\ automorphism} $\varphi \in$ {\rm Aut}($E^{[n]}$), {\rm the\ condition} $\varphi ^{\ast }({\mathcal O}_{E^{[n]}}(D_{E}))={\mathcal O}_{E^{[n]}}(D_{E})$ {\rm is\ equivalent\ to\ the\ condition} $\varphi (D_{E})=D_{E}$.
\end{mar}

Recall that $\pi \circ \omega :K^{n}\setminus \Gamma _{K}\rightarrow E^{[n]}\setminus D_{E}$ is the universal covering space.
\begin{thm}\label{thm:2}
Let $E$ be an Enriques surface, $D_{E}$ the exceptional divisor of the Hilbert-Chow morphism $\pi _{E}:E^{[n]}\rightarrow E^{(n)}$. An automorphism $f$ of $E^{[n]}$ is natural if and only if $f(D_{E})=D_{E}$, i.e. $f^{\ast }({\mathcal O}_{E^{[n]}}(D_{E}))={\mathcal O}_{E^{[n]}}(D_{E})$. 
\end{thm}
\begin{proof}
Let $f$ be an automorphism of $E^{[n]}$ with $f(D_{E})=D_{E}$.
Then $f$ induces an automorphism of $E^{[n]}\backslash D_{E}$.
Since the uniqueness of the universal covering space, there is an automorphism $g$ of $K^{n}\backslash \Gamma _{K}$ such that $\pi \circ \omega \circ g=f\circ \pi \circ \omega $:
$$
\xymatrix{
K^{n}\setminus \Gamma _{K} \ar[d]^{\pi \circ \omega } \ar[r]^{g} &K^{n}\setminus \Gamma _{K} \ar[d]^{\pi \circ \omega } \\
E^{[n]}\setminus D_{E} \ar[r]^{f} &E^{[n]}\setminus D_{E}.
}
$$
Since $\Gamma _{K}$ is an analytic set of codimension $2$, and $K^{n}$ is projective, $g$ can be extended to a birational automorphism of $K^{n}$. By Oguiso $[\ref{bio:4},\,{\rm Proposition}\,4.1]$,
$g$ is an automorphism of $K^{n}$, and 
there are some automorphisms $g_{1},\ldots,g_{n}\in $ Aut$(K)$ and $s\in {\mathcal S}_{n}$ such that $g=s\circ g_{1}\times \cdots\times g_{n}$. Since ${\mathcal S}_{n}\subset G$, we can assume that $g=g_{1}\times \cdots\times g_{n}$.
\\

Recall that 
we denote the covering transformation group of $\pi \circ \omega $ by:  
\[ G:=\{g\in {\rm Aut}(K^{n}\setminus \Gamma _{K}):\pi \circ \omega \circ g=\pi \circ \omega \}. \] 
By Proposition $\ref{dfn:12}$ below,
we have $g_{i}=g_{1}$ or $g_{1}\circ \sigma $ for $1\leq i\leq n$ and $g_{1}\circ \sigma =\sigma \circ g_{1}$.
We denote $g_{1}^{[n]}$ the induced automorphism of $E^{[n]}$ given by $g_{1}$. Then ${g_{1}}^{[n]}|_{E^{[n]}\backslash D_{E}}=f|_{E^{[n]}\backslash D_{E}}$. Thus $g_{1}^{[n]}=f$, so $f$ is natural. The other implication is obvious.
\end{proof}
\begin{pro}\label{dfn:12}
In the proof of Theorem $4.3$, we have $g_{i}=g_{1}$ or $g_{i}=g_{1}\circ \sigma $ for each $1\leq i\leq n$. Moreover $g_{1}\circ \sigma =\sigma \circ g_{1}$.
\end{pro}
\begin{proof}
We show the first assertion by contradiction.
Without loss of generality, we may assume that $g_{2}\not=g_{1}$ and $g_{2}\not=g_{1}\circ \sigma $.
Let $h_{1}$ and $h_{2}$ be two morphisms of $K$ where $g_{i}\circ h_{i}={\rm id}_{K}$ and $h_{i}\circ g_{i}={\rm id}_{K}$ for $i=1$, $2$. 
We define two morphisms $H_{1,2}$ and $H_{1,2,\sigma }$ from $K$ to $K^{2}$ by following.
\[ H_{1,2}:K\ni x\mapsto (h_{1}(x),h_{2}(x))\in K^{2} \]
\[ H_{1,2,\sigma }:K\ni x\mapsto (h_{1}(x),\sigma \circ h_{2}(x))\in K^{2}. \]
Let $S_{\sigma }:=\{ (x,y) |\,y=\sigma (x)\}$ be the subset of $K^{2}$.
Since $h_{1}\not=h_{2}$ and $h_{1}\not=\sigma \circ h_{2}$, $H_{1,2}^{-1}(\Delta _{K}^{2})\cup H_{1,2,\sigma }^{-1}(S_{\sigma })$ do not coincide $K$. Thus there is $x'\in K$ such that $H_{1,2}(x')\not\in \Delta _{K}^{2}$ and $H_{1,2,\sigma }(x')\not\in S_{\sigma }$. 
For $x'\in K$, we put $x_{i}:=h_{i}(x')\in K$ for $i=1$, $2$. Then there are some elements $x_{3},\ldots ,x_{n}\in K$ such that $(x_{1},\ldots ,x_{n})\in K^{n}\setminus \Gamma _{K}$.  
We have $g((x_{1},\ldots ,x_{n}))\not\in K^{n}\backslash \Gamma _{K}$ by the assumption of $x_{1}$ and $x_{2}$. It is contradiction, because $g$ is an automorphism of $K^{n}\backslash \Gamma _{K}$. Thus we have $g_{i}=g_{1}$ or $g_{i}=g_{1}\circ \sigma $ for $1\leq i\leq n$.

We show the second assertion.
Since the covering transformation group of $\pi \circ \omega $ is $G$,
the liftings of $f$ are given by 
\[ \{ g\circ u: u\in G \}=\{ u\circ g: u\in G \}. \]
Thus for $\sigma _{1}\circ g$, there is an element $\sigma _{i_{1}\cdots i_{k}}\circ s$ of $G$ where $s\in {\mathcal S}_{n}$ and $t\in \{\sigma _{i_{1}\ldots i_{k}}\}_{1\leq k\leq n,\ 1\leq i_{1}<\ldots <i_{k}\leq n}$ such that $\sigma _{1}\circ g=g\circ \sigma _{i_{1}\cdots i_{k}}\circ s$.
If we think about the first component of $\sigma _{1}\circ g$ and $[\ref{bio:9},\,{\rm Lemma}\,1.2]$, we have $s={\rm id}$ and $t=\sigma _{1}$.
Therefore $g\circ \sigma _{1}\circ g^{-1}=\sigma _{1}$, we have $\sigma \circ g_{1}=g_{1}\circ \sigma $.
\end{proof}

\section{Proof of Theorem $1.3$}
Let $E$ be an Enriques surface, $E^{[n]}$ the Hilbert scheme of $n$ points of $E$, and $\pi :X\rightarrow E^{[n]}$ the universal covering space of $E^{[n]}$ where $X$ is a Calabi-Yau manifold.
First, for $n=2$, we compute the Hode number of $X$.
Next, for $n\geq 3$, we show that the covering involution of $\pi :X\rightarrow E^{[n]}$ acts on $H^{2}(X,{\mathbb C})$ as identity, and by using Theorem $1.2$, we classify automorphisms of $X$ acting on $H^{2}(X,{\mathbb C})$ identically and its order is $2$.
Finally, we show Theorem $1.3$.
\\

We suppose $n=2$. 
Since $E^{2}_{\ast }=E^{2}$, we have
$E^{[2]}=E^{[2]}_{\ast }=$ Blow$_{\Delta _{E}^{2}}E^{2}/{\mathcal S}_{2}$.
Let $\pi :X\rightarrow E^{[2]}$ be the universal covering space of $E^{[2]}$.
Since $K^{2}_{\ast \mu }=K^{2}$ and Proposition $\ref{dfn:30}$, we have 
\[ X\simeq {\rm Blow}_{\Delta _{K}^{2}\cup T}K^{2}/H, \]
where $T:=\{(x,y)\in K^{2}:y=\sigma (x)\}$.
Let $\eta:{\rm Blow}_{\Delta _{K}^{2}\cup T}K^{2}/H\rightarrow K^{2}/H$ be the natural map.
We put 
\[ D_{\Delta }:=\eta ^{-1}(\Delta _{K}^{2}/H)\ {\rm and}\]  
\[ D_{T}:=\eta ^{-1}(T/H). \] 
For two inclusions 
\[ j_{D_\Delta }:D_{\Delta }\hookrightarrow {\rm Blow}_{\Delta _{K}^{2}\cup T}K^{2}/H,\ {\rm and} \] 
\[ j_{D_{T}}:D_{T}\hookrightarrow {\rm Blow}_{\Delta _{K}^{2}\cup T}K^{2}/H, \]
let $j_{\ast D_{\Delta }}$ be the Gysin morphism
\[ j_{\ast D_{\Delta }}:H^{p}(D_{\Delta },{\mathbb C})\rightarrow H^{p+2}({\rm Blow}_{\Delta _{K}^{2}\cup T}K^{2}/H,{\mathbb C}), \]
$j_{\ast D_{T}}$ the Gysin morphism
\[ j_{\ast D_{T}}:H^{p}(D_{T},{\mathbb C})\rightarrow H^{p+2}({\rm Blow}_{\Delta _{K}^{2}\cup T}K^{2}/H,{\mathbb C}), {\rm and} \]
\[ \psi:=\eta ^{\ast }+j_{\ast D_{\Delta }}\circ \eta |_{D_{\Delta }}^{\ast }+j_{\ast D_{T}}\circ \eta |_{D_{T}}^{\ast } \]
morphisms from $H^{p}(K^{2}/H,{\mathbb C})\oplus H^{p-2}(\Delta _{K}^{2}/H,{\mathbb C})\oplus H^{p-2}(T/H,{\mathbb C})$ to\\
$H^{p}({\rm Blow}_{\Delta _{K}^{2}\cup T}K^{2}/H,{\mathbb C}).$
From $[\ref{bio:30},\,{\rm Theorem}\,7.31],$ we have isomorphisms of Hodge structure on
$H^{k}({\rm Blow}_{\Delta _{K}^{2}\cup T}K^{2}/H,{\mathbb C})$ by $\psi $:
\begin{equation}
\label{eq:20}
H^{k}(K^{2}/H,{\mathbb C})\oplus H^{k-2}(\Delta _{K}^{2}/H,{\mathbb C})\oplus H^{k-2}(T/H,{\mathbb C})\simeq H^{k}({\rm Blow}_{\Delta _{K}^{2}\cup T}K^{2}/H,{\mathbb C}).
\end{equation}
For algebraic variety $Y$,
let $h^{p,q}(Y)$ be the number $h^{p,q}(Y)=$ dim$_{{\mathbb C}}H^{p+q}(Y,{\mathbb C})^{p,q}$.
\begin{thm}\label{dfn:91}
For the universal covering space $\pi :X\rightarrow E^{[2]}$, we have $h^{0,0}(X)=1$, $h^{1,0}(X)=0$, $h^{2,0}(X)=0$, $h^{1,1}(X)=12$,
$h^{3,0}(X)=0$, $h^{2,1}(X)=0$, $h^{4,0}(X)=1$, $h^{3,1}(X)=10$, and $h^{2,2}(X)=131$.
\end{thm}
\begin{proof}
Let $\sigma $ be the covering involution of $\mu :K\rightarrow E$. Put
\[ H^{k}_{\pm }(K,{\mathbb C})^{p,q}:=\{ \alpha \in H^{k}(K,{\mathbb C})^{p,q}:\sigma ^{\ast }(\alpha )=\pm \alpha  \}\,{\rm and}  \] 
\[ h^{p,q}_{\pm }(K):={\rm dim}_{\mathbb C}H^{k}_{\pm }(K,{\mathbb C})^{p,q}.  \]
Then for an Enriques surface $E\simeq K/\langle\sigma \rangle$, we have
\[ H^{k}(E,{\mathbb C})^{p,q}\simeq H^{k}_{+}(K,{\mathbb C})^{p,q}. \]
Since $K$ is a $K3$ surface, we have 
\[ h^{0,0}(K)=1,\ h^{1,0}(K)=0,\ h^{2,0}(K)=1,\ {\rm and}\ h^{1,1}(K)=20,\,{\rm and} \]
\[ h_{+}^{0,0}(K)=1,\,h_{+}^{1,0}(K)=0,\,h_{+}^{2,0}(K)=0,\,{\rm and}\,h_{+}^{1,1}(K)=10,\,{\rm and} \]
\[ h^{0,0}_{-}(K)=0,\,h^{1,0}_{-}(K)=0,\,h^{2,0}_{-}(K)=1,\,{\rm and}\,h^{2,0}_{-}(K)=10. \]
Since $n=2$,
we obtain $\Delta _{K}^{2}/H\simeq E$ and $T/H\simeq E$.
Thus we have
\[ h^{0,0}(\Delta _{K}^{2}/H)=1,\,h^{1,0}(\Delta _{K}^{2}/H)=0,\,h^{2,0}(\Delta _{K}^{2}/H)=0,\,{\rm and}\,h^{1,1}(\Delta _{K}^{2}/H)=10, \]
and we have
\[ h^{0,0}(T/H)=1,\,h^{1,0}(T/H)=0,\,h^{2,0}(T/H)=0,\,{\rm and}\,h^{1,1}(T/H)=10. \]
By the definition of $H$, we obtain $H=\langle{\mathcal S}_{2},\sigma _{1,2}\rangle$.
From the K\"unneth Theorem, we have 
\[ H^{p+q}(K^{2},{\mathbb C})^{p,q}\simeq \bigoplus _{s+u=p,t+v=q}H^{s+t}(K,{\mathbb C})^{s,t}\otimes H^{u+v}(K,{\mathbb C})^{u,v},\,{\rm and} \]
\[ H^{k}(K^{2}/H,{\mathbb C})^{p,q}\simeq \{ \alpha \in H^{k}(K^{2},{\mathbb C})^{p,q}:s^{\ast }(\alpha )=\alpha \,{\rm for}\,s\in {\mathcal S}_{2}\,{\rm and}\,\sigma ^{\ast }_{1,2}(\alpha )=\alpha \}. \]
Thus we obtain
\[ h^{0,0}(K^{2}/H)=1,\,h^{1,0}(K^{2}/H)=0,\,h^{2,0}(K^{2}/H)=0,\,h^{1,1}(K^{2}/H)=10,\]
\[ h^{3,0}(K^{2}/H)=0,\,h^{2,1}(K^{2}/H)=0,\,h^{4,0}(K^{2}/H)=1, \]
\[ h^{3,1}(K^{2}/H)=10,\, {\rm and}\,h^{2,2}(K^{2}/H)^{2,2}=111. \]
Specially, we fix a basis $\beta $ of $H^{2}(K,{\mathbb C})^{2,0}$ and a basis $\{\gamma _{i}\}_{i=1}^{10}$ of $H^{2}_{-}(K,{\mathbb C})^{1,1}$,
then we have 
\begin{equation}
\label{eq:30}
H^{4}(K^{2}/H,{\mathbb C})^{3,1}\simeq \bigoplus _{i=1}^{10}{\mathbb C}(\beta \otimes \gamma _{i}+\gamma _{i}\otimes \beta ).
\end{equation}
By the above equation $(\ref{eq:20})$, we have
\[ h^{0,0}({\rm Blow}_{\Delta _{K}^{2}\cup T}K^{2}/H)=1,\,h^{1,0}({\rm Blow}_{\Delta _{K}^{2}\cup T}K^{2}/H)=0,\]
\[ h^{2,0}({\rm Blow}_{\Delta _{K}^{2}\cup T}K^{2}/H)=0,\,h^{1,1}({\rm Blow}_{\Delta _{K}^{2}\cup T}K^{2}/H)=12,\]
\[ h^{3,0}({\rm Blow}_{\Delta _{K}^{2}\cup T}K^{2}/H)=0,\,h^{2,1}({\rm Blow}_{\Delta _{K}^{2}\cup T}K^{2}/H)=0, \]
\[ h^{4,0}({\rm Blow}_{\Delta _{K}^{2}\cup T}K^{2}/H)=1,\,h^{3,1}({\rm Blow}_{\Delta _{K}^{2}\cup T}K^{2}/H)=10,\,{\rm and}\]
\[ h^{2,2}({\rm Blow}_{\Delta _{K}^{2}\cup T}K^{2}/H)=131. \]
Thus we obtain
$h^{0,0}(X)=1$, $h^{1,0}(X)=0$, $h^{2,0}(X)=0$, $h^{1,1}(X)=12$,
$h^{3,0}(X)=0$, $h^{2,1}(X)=0$, $h^{4,0}(X)=1$, $h^{3,1}(X)=10$, and $h^{2,2}(X)=131$.
\end{proof}

We show that for $n\geq 3$, the covering involution of $\pi :X\rightarrow E^{[n]}$ acts on $H^{2}(X,{\mathbb C})$ as identity, by using Theorem $1.2$ we classify automorphisms of $X$ acting on $H^{2}(X,{\mathbb C})$ identically and its order is $2$, and Theorem $1.3$ from here.
\begin{lem}\label{dfn:90}
Let $X$ be a smooth complex manifold, $Z\subset X$ a closed submanifold with codimension is $2$, $\tau :X_{Z}\rightarrow X$ the blow up of $X$ along $Z$, $E=\tau ^{-1}(Z)$ the exceptional divisor, and $h$ the first Chern class of the line bundle ${\mathcal O}_{X_{Z}}(E)$.\\
Then $\tau ^{\ast }:H^{2}(X,{\mathbb C})\rightarrow H^{2}(X_{Z},{\mathbb C})$ is injective, and 
\[ H^{2}(X_{Z},{\mathbb C})\simeq H^{2}(X,{\mathbb C})\oplus {\mathbb C}h. \]
\end{lem}
\begin{proof}
Let $U:=X\setminus Z$ be an open set of $X$. Then $U$ is isomorphic to an open set $U'=X_{Z}\setminus E$ of $X_{Z}$.
As $\tau $ gives a morphism between the pair $(X_{Z},U')$ and the pair $(X,U)$, we have a morphism $\tau ^{\ast }$ between the long exact sequence of cohomology relative to these pairs:
$$
\xymatrix{
H^{k}(X,U,{\mathbb C}) \ar[r] \ar[d]^{\tau _{X,U}^{\ast }} &H^{k}(X,{\mathbb C}) \ar[r] \ar[d]^{\tau ^{\ast }_{X}} &H^{k}(U,{\mathbb C}) \ar[d]^{\tau _{U}^{\ast }} \ar[r] &H^{k+1}(X,U,{\mathbb C}) \ar[d]^{\tau _{X,U}^{\ast }} \\
H^{k}(X_{Z},U',{\mathbb C}) \ar[r] &H^{k}(X_{Z},{\mathbb C}) \ar[r] &H^{k}(U',{\mathbb C}) \ar[r] &H^{k+1}(X_{Z},U',{\mathbb C}).  
}
$$
By Thom isomorphism, the tubular neighborhood Theorem, and Excision theorem,
we have  
\[ H^{q}(Z,{\mathbb C})\simeq H^{q+4}(X,U,{\mathbb C}),\ {\rm and}\]
\[ H^{q}(E,{\mathbb C})\simeq H^{q+2}(X_{Z},U',{\mathbb C}). \]
In particular, we have 
\[ H^{l}(X,U,{\mathbb C})=0\ {\rm for}\ l=0,1,2,3,\ {\rm and} \]
\[ H^{j}(X_{Z},U',{\mathbb C})=0\ {\rm for}\ l=0,1. \]
Thus we have
$$
\xymatrix{
0 \ar[r] \ar[d]^{\tau _{X,U}^{\ast }} &H^{1}(X,{\mathbb C}) \ar[r] \ar[d]^{\tau ^{\ast }_{X}} &H^{1}(U,{\mathbb C}) \ar[d]^{\tau _{U}^{\ast }} \ar[r] &0 \ar[d]^{\tau _{X,U}^{\ast }} \\
0 \ar[r] &H^{1}(X_{Z},{\mathbb C}) \ar[r] &H^{1}(U',{\mathbb C}) \ar[r] &H^{0}(E,{\mathbb C}),
}
$$
and 
$$
\xymatrix{
0 \ar[r] \ar[d]^{\tau _{X,U}^{\ast }} &H^{2}(X,{\mathbb C}) \ar[r] \ar[d]^{\tau ^{\ast }_{X}} &H^{2}(U,{\mathbb C}) \ar[d]^{\tau _{U}^{\ast }} \ar[r] &0 \ar[d]^{\tau _{X,U}^{\ast }} \\
H^{0}(E,{\mathbb C}) \ar[r] &H^{2}(X_{Z},{\mathbb C}) \ar[r] &H^{2}(U',{\mathbb C}) \ar[r] &H^{3}(X_{Z},U',{\mathbb C}).
}
$$

Since $\tau \mid _{U'}:U'\xrightarrow{\sim}U$, we have isomorphisms $\tau _{U}^{\ast }:H^{k}(U,{\mathbb C})\simeq H^{k}(U',{\mathbb C})$.
Thus we have
\[ {\rm dim}_{{\mathbb C}}H^{2}(X_{Z},{\mathbb C})={\rm dim}_{{\mathbb C}}H^{2}(X,{\mathbb C})+1,\ {\rm and} \] 
\[ \tau ^{\ast }:H^{2}(X,{\mathbb C})\rightarrow H^{2}(X_{Z},{\mathbb C})\ {\rm is\ injective},\]
and therefore we obtain
\[ H^{2}(X_{Z},{\mathbb C})\simeq H^{2}(X,{\mathbb C})\oplus {\mathbb C}h. \]
\end{proof}
\begin{pro}\label{pro:70}
Suppose $n\geq 3$. For the universal covering space $\pi :X\rightarrow E^{[n]}$,  ${\rm dim}_{{\mathbb C}}H^{2}(X,{\mathbb C})=11$.
\end{pro}
\begin{proof}
Since the codimension of $\pi ^{-1}(F_{E})$ is $2$, $H^{2}(X,{\mathbb C})\cong H^{2}(X\setminus \pi ^{-1}(F_{E}),{\mathbb C})$.
By Proposition $2.6$, $X\setminus \pi ^{-1}(F_{E})\simeq {\rm Blow}_{\Delta _{K\ast \mu }\cup T_{\ast \mu }}K^{n}_{\ast \mu }/H$.
\\
Let $\tau :{\rm Blow}_{\Delta _{K\ast \mu }\cup T_{\ast \mu }}K^{n}_{\ast \mu }\rightarrow K^{n}_{\ast \mu }$ be the blow up of $K^{n}_{\ast \mu }$ along $\Delta _{K\ast \mu }\cup T_{\ast \mu }$, \\
\[ h_{ij}\ {\rm the\ first\ Chern\ class\ of\ the\ line\ bundle}\ {\mathcal O}_{{\rm Blow}_{\Delta _{K\ast \mu }\cup T_{\ast \mu }}K^{n}_{\ast \mu }}(\tau ^{-1}(\Delta _{K\ast \mu \,i,})), \]
and
\[ k_{ij}\ {\rm the\ first\ Chern\ class\ of\ the\ line\ bundle}\ {\mathcal O}_{{\rm Blow}_{\Delta _{K\ast \mu }\cup T_{\ast \mu }}K^{n}_{\ast \mu }}(\tau ^{-1}(T _{K\ast \mu \,ij})).\]
By Lemma $\ref{dfn:90}$, we have
\[ H^{2}({\rm Blow}_{\Delta _{K\ast \mu }\cup T_{\ast \mu }}K^{n}_{\ast \mu },{\mathbb C})\cong H^{2}(K^{n},{\mathbb C})\oplus \biggl{(}\bigoplus _{1\leq i<j\leq n}{\mathbb C}h_{ij}\biggl{)} \oplus \biggl{(}\bigoplus _{1\leq i<j\leq n}{\mathbb C}k_{ij}\biggl{)}. \]
Since $n\geq 3$, there is an isomorphism 
\[ (j,j+1)\circ \sigma _{ij}\circ (j,j+1):\bigtriangleup _{K\ast \mu \,ij}\xrightarrow{\sim}T_{\ast \mu \,ij}. \]
Thus we have dim$_{{\mathbb C}}H^{2}({\rm Blow}_{\Delta _{K\ast \mu }\cup T_{\ast \mu }}K^{n}_{\ast \mu }/H,{\mathbb C})=11$, 
i.e.\,${\rm dim}_{{\mathbb C}}H^{2}(X,{\mathbb C})=11$.
\end{proof}
\begin{pro}\label{dfn:13}
${\rm dim}_{{\mathbb C}}H^{0}(X,{\mathcal O}_{X}(\pi ^{\ast }(D_{E})))=1$.
\end{pro}
\begin{proof}
Since $\pi $ is finite, we obtain ${\rm dim}_{{\mathbb C}}H^{0}(X,{\mathcal O}_{X}(\pi ^{\ast }(D_{E})))={\rm dim}_{{\mathbb C}}H^{0}(E^{[n]},\pi _{\ast }{\mathcal O}_{X}(\pi ^{\ast }(D_{E})))$.
From the projective formula and $X\simeq {\mathcal Spec}\,{\mathcal O}_{E^{[n]}}\oplus {\mathcal O}_{E^{[n]}}(K_{E^{[n]}})$, we have $\pi _{\ast }{\mathcal O}_{X}(\pi ^{\ast }(D_{E}))\simeq {\mathcal O}_{E^{[n]}}(D_{E})\oplus {\mathcal O}_{E^{[n]}}(D_{E}\otimes K_{E^{[n]}})$.
By Proposition $\ref{dfn:80}$, dim$_{{\mathbb C}}H^{0}(E^{[n]},{\mathcal O}_{E^{[n]}}(D_{E}))=1$.\\
We show that 
\[ {\rm dim}_{{\mathbb C}}H^{0}(E^{[n]},{\mathcal O}_{E^{[n]}}(D_{E}\otimes K_{E^{[n]}}))=0. \]
Since $\pi _{E}|_{E^{[n]}\setminus D_{E}}:E^{[n]}\setminus D_{E}\simeq E^{(n)}\setminus \Delta _{E}^{(n)}$, we have 
\[ (\pi _{E})_{\ast }({\mathcal O}_{E^{[n]}}(D_{E}\otimes K_{E^{[n]}}))\simeq \Omega ^{2n}_{E^{(n)}}\ {\rm on}\ E^{(n)}\setminus \Delta _{E}^{(n)}.\]
Hence we have 
\[ \Gamma (E^{[n]}\setminus D_{E},{\mathcal O}_{E^{[n]}}(D_{E}\otimes K_{E^{[n]}}))\simeq \Gamma (E^{(n)}\setminus \Delta _{E}^{(n)},\Omega ^{2n}_{E^{(n)}}). \]
Since $H^{2}(E,{\mathbb C})^{2,0}=0$, 
and from the K\"unneth Theorem,
\[ H^{2n}(E^{n},{\mathbb C})^{2n,0}=H^{0}(E^{n},\Omega ^{2n}_{E^{n}})=0. \]
Since the codimension of $\Delta _{E}^{n}$ is $2$, and $\Omega ^{2n}_{E^{n}}$ is a locally free sheaf, we have 
\[ \Gamma (E^{n}\setminus \Delta _{E}^{n},\Omega ^{2n}_{E^{n}})=H^{0}(E^{n},\Omega ^{2n}_{E^{n}}). \]
Thus we have 
\[ \Gamma (E^{(n)}\setminus \Delta _{E}^{(n)},\Omega ^{2n}_{E^{(n)}})=0, \]
and therefore
\[ {\rm dim}_{{\mathbb C}}H^{0}(E^{[n]}\setminus D_{E},{\mathcal O}_{E^{[n]}}(D_{E}\otimes K_{E^{[n]}}))=0. \]
Hence 
\[ {\rm dim}_{{\mathbb C}}H^{0}(E^{[n]},{\mathcal O}_{E^{[n]}}(D_{E}\otimes K_{E^{[n]}}))=0. \]
Thus we obtain ${\rm dim}_{{\mathbb C}}H^{0}(X,{\mathcal O}_{X}(\pi ^{\ast }(D_{E})))=1$.
\end{proof}
\begin{mar}\label{dfn:1002}
{\rm Then\ by} {\rm Proposition}\,$\ref{dfn:13}$, {\rm for\ an\ automorphism} $\varphi \in $ {\rm Aut}(X), {\rm the\ condition} $\varphi ^{\ast }({\mathcal O}_{X}(\pi ^{\ast }D_{E}))={\mathcal O}_{X}(\pi ^{\ast }D_{E})$ {\rm is\ equivalent\ to\ the condition} $\varphi (\pi ^{-1}(D_{E}))=\pi ^{-1}(D_{E})$.
\end{mar}
Let $\rho $ be the covering involution of $\pi :X\rightarrow E^{[n]}$. 
\begin{pro}\label{pro:3}
For $n\geq 3$, the induced map $\rho ^{\ast }:H^{2}(X,{\mathbb C})\rightarrow H^{2}(X,{\mathbb C})$ is identity.
\end{pro}
\begin{proof}
Since $E^{[n]}\simeq X/\langle \rho \rangle$ , we have $H^{2}(E^{[n]},{\mathbb C})\simeq H^{2}(X,{\mathbb C})^{\rho ^{\ast }}$.
By Proposition $\ref{pro:70}$, for $n\geq 3$, we have dim$_{\mathbb C}H^{2}(X,{\mathbb C})=11$. 
By $[\ref{bio:6},\,{\rm page}\,767]$, dim$_{\mathbb C}H^{2}(E^{[n]},{\mathbb C})=11$.
Thus the induced map $\rho ^{\ast }:H^{2}(X,{\mathbb C})\rightarrow H^{2}(X,{\mathbb C})$ is identity for $n\geq 3$. 
\end{proof}
Recall that $\mu :K\rightarrow E$ is the universal covering of $E$ where $K$ is a $K3$ surface, and $\sigma $ the covering involution of $\mu $.
\begin{pro}\label{dfn:130}
Let $E$ be an Enriques surface which does not have numerically trivial involutions, $E^{[n]}$ the Hilbert scheme of $n$ points of $E$, $\pi :X\rightarrow E^{[n]}$ the universal covering space of $E^{[n]}$, $\rho $ the covering involution of $\pi $, and $n\geq 3$. 
Let $\iota $ be an involution of $X$ which acts on $H^{2}(X,{\mathbb C})$ as ${\rm id}$, then $\iota =\rho $.
\end{pro}
\begin{proof}
Let $\iota $ be an involution of $X$ which acts on $H^{2}(X,{\mathbb C})$ as ${\rm id}$.
By Remark $\ref{dfn:1002}$, 
$\iota |_{X\setminus \pi ^{-1}(D_{E})}$ is automorphism of $X\setminus \pi ^{-1}(D_{E})$.
By the uniqueness of the universal covering space, there is an automorphism $g$ of $K^{n}\backslash \Gamma _{K}$ 
such that $\iota \circ \omega =\omega \circ g$:
$$
\xymatrix{
K^{n}\setminus \Gamma _{K} \ar[d]^{\omega } \ar[r]^{g} &K^{n}\setminus \Gamma _{K} \ar[d]^{\omega } \\
X\setminus \pi ^{-1}(D_{E}) \ar[r]^{\iota }&X\setminus \pi ^{-1}(D_{E}).
}
$$
Like the proof of Proposition $\ref{dfn:12}$, we can assume that there are some automorphisms $g_{i}$ of $K$ such that $g=g_{1}\times \cdots \times g_{n}$, for each $1\leq i\leq n$, $g_{i}=g_{1}$ or $g_{i}=g_{1}\circ \sigma $, and $g_{1}\circ \sigma =\sigma \circ g_{1}$.
Since $\iota ^{2}={\rm id}_{X}$, so we have $g^{2}\in H$. Thus we have $g^{2}={\rm id}_{K^{n}}$ or $\sigma _{i_{1}\ldots i_{k}}$. By $[\ref{bio:9},\,{\rm Lemma}\,1.2]$,  
we have $g^{2}={\rm id}_{K^{n}}$.
We put $g':=g_{1}$.
Let $g'_{E}$ be the induced automorphism of $E$ by $g'$, and $g'^{[n]}_{E}$ the induced automorphism of $E^{[n]}$ by $g'_{E}$.
Since $g'^{[n]}_{E}\circ \pi =\pi \circ \iota $ and $n\geq 3$,
$g'^{[n]\ast }_{E}$ acts on $H^{2}(E^{[n]},{\mathbb C})$ as ${\rm id}$, and therefore $g'^{\ast }_{E}$ acts on $H^{2}(E,{\mathbb C})$ as ${\rm id}$.
Since $E$ does not have numerically trivial involutions, $g'_{E}={\rm id}_{E}$, and therefore we have $g'=\sigma $ or $g'={\rm id}_{K}$.
Thus we have $\pi \circ \omega \circ g=\pi \circ \omega $:
$$
\xymatrix{
K^{n}\setminus \Gamma _{K} \ar[d]^{\pi \circ \omega } \ar[r]^{g} &K^{n}\setminus \Gamma _{K} \ar[d]^{\pi \circ \omega } \\
E^{[n]}\setminus D_{E} \ar[r]^{{\rm id}}&E^{[n]}\setminus D_{E}.
}
$$
Since $\iota \circ \omega =\omega \circ g$, we have we have $\pi =\pi \circ \iota $:
$$
\xymatrix{
X\setminus \pi ^{-1}(D_{E}) \ar[d]^{\pi } \ar[r]^{\iota } &X\setminus \pi ^{-1}(D_{E}) \ar[d]^{\pi } \\
E^{[n]}\setminus D_{E} \ar[r]^{{\rm id}}&E^{[n]}\setminus D_{E}.
}
$$
Since the degree of $\pi$ is $2$, we have $\iota =\rho $. 
\end{proof}
We suppose that $E$ has numerically trivial involutions.
By $[\ref{bio:9},\,{\rm Proposition}\,1.1]$, 
there is just one automorphism of $E$, denoted $\upsilon $, such that its order is $2$, and $\upsilon ^{\ast }$ acts on $H^{2}(E,{\mathbb C})$ as ${\rm id}$.
For $\upsilon $, there are just two involutions of $K$ which are liftings of $\upsilon $,
one acts on $H^{0}(K,\Omega ^{2}_{K})$ as ${\rm id}$, 
and another acts on $H^{0}(K,\Omega ^{2}_{K})$ as $-{\rm id}$, we denote by $\upsilon _{+}$ and $\upsilon _{-}$, respectively. 
Then they satisfies $\upsilon _{+}=\upsilon _{-}\circ \sigma $.
Let $\upsilon ^{[n]}$ be the automorphism of $E^{[n]}$ which is induced by $\upsilon $.
For $\upsilon ^{[n]}$, there are just two automorphisms of $X$ which are liftings of $\upsilon ^{[n]}$, denoted $\varsigma $ and $\varsigma '$, respectively:
$$
\xymatrix{
X \ar[d]^{\pi } \ar[r]^{\varsigma \,(\varsigma ')} &X \ar[d]^{\pi }\\
E^{[n]} \ar[r]^{\upsilon ^{[n]}} &E^{[n]}.
}
$$
Then they satisfies $\varsigma =\varsigma '\circ \sigma $.
Since $n\geq 3$ and like the proof of Proposition $\ref{dfn:130}$, each order of $\varsigma $ and $\varsigma '$ is $2$ .
\begin{lem}\label{dfn:100}
For $\varsigma $ and $\varsigma '$, one acts on $H^{0}(X,\Omega ^{2n}_{X})$ as ${\rm id}$, and another act on $H^{0}(X,\Omega ^{2n}_{X})$ as $-{\rm id}$.
\end{lem}
\begin{proof}
Since $\upsilon ^{[n]}|_{E^{[n]}\setminus D_{E}}$ is an automorphism of $E^{[n]}\setminus D_{E}$, and
from the uniqueness of the universal covering space,
there is an automorphism $g$ of $K^{n}\setminus \Gamma _{K}$ such that $\upsilon ^{[n]} \circ \pi \circ \omega =\pi \circ \omega \circ g$:
$$
\xymatrix{
K^{n}\setminus \Gamma _{K} \ar[d]^{\pi \circ \omega } \ar[r]^{g} &K^{n}\setminus \Gamma _{K} \ar[d]^{\pi \circ \omega } \\
E^{[n]}\setminus D_{E} \ar[r]^{\upsilon ^{[n]}} &E^{[n]}\setminus D_{E}.
}
$$
Like the proof of Proposition $\ref{dfn:12}$, we can assume that there are some automorphisms $g_{i}$ of $K$ such that $g=g_{1}\times \cdots\times g_{n}$ for each $1\leq i\leq n$, $g_{i}=g_{1}$ or $g_{i}=g_{1}\circ \sigma $, and $g_{1}\circ \sigma =\sigma \circ g_{1}$.
From Theorem $\ref{thm:10}$, we get $K^{n}\setminus \Gamma _{K}/H\simeq X\setminus \pi ^{-1}(D_{E}).$
Put 
\[ \upsilon _{+,even}:=u_{1}\times \cdots \times u_{n}  \]
where 
\[ u_{i}=\upsilon _{+}\ {\rm or}\ u_{i}=\upsilon _{-}\ {\rm and}\ {\rm the\ number\ of}\ i\ {\rm with}\ u_{i}=\upsilon _{+}\ {\rm is\ even} \]
which is an automorphism of $K^{n}$ and induces an automorphism $\widetilde{\upsilon _{+,even}}$ of $X\setminus \pi ^{-1}(D_{E}).$
We define automorphisms $\widetilde{\upsilon _{+,odd}}$, $\widetilde{\upsilon _{-,even}}$, and $\widetilde{ \upsilon _{-,odd}}$
of $K^{n}\setminus \Gamma _{K}/H$ in the same way.
Since $\sigma _{ij}\in H$ for $1\leq i<j\leq n$, and $\upsilon _{+}=\upsilon _{-}\circ \sigma $,
if $n$ is odd,
\[ \widetilde{\upsilon _{+,odd}}=\widetilde{\upsilon _{-,even}},\ \widetilde{\upsilon _{+,even}}=\widetilde{\upsilon _{-,odd}},\ {\rm and}\ \widetilde{\upsilon _{+,odd}}\not=\widetilde{\upsilon _{+,even}}, \]
and if $n$ is even, 
\[ \widetilde{\upsilon _{+,odd}}=\widetilde{\upsilon _{-,odd}},\ \widetilde{\upsilon _{+,even}}=\widetilde{\upsilon _{-,even}},\ {\rm and}\ \widetilde{\upsilon _{+,odd}}\not=\widetilde{\upsilon _{+,even}}. \]
Since $\upsilon ^{[n]}\circ \pi =\pi \circ \widetilde{\upsilon _{+,odd}}$ and $\upsilon ^{[n]}\circ \pi =\pi \circ \widetilde{\upsilon _{+,even}}$, and the degree of $\pi $ is $2$,
Thus we have $\{\varsigma ,\varsigma '\}=\{\widetilde{\upsilon _{+,odd}},\widetilde{\upsilon _{+,even}}\}$.
 
Let $\omega _{X}\in H^{0}(X,\Omega ^{2n}_{X})$ be a basis of $H^{0}(X,\Omega ^{2n}_{X})$ over ${\mathbb C}$. Since $X\setminus \pi ^{-1}(F_{E})\simeq {\rm Blow}_{\Delta _{K\ast \mu }\cup T_{\ast \mu }}K^{n}_{\ast \mu }/H$, and
by the definition of $\upsilon _{+}$ and $\upsilon _{-}$,
\[ \widetilde{\upsilon _{+,odd}}^{\ast }(\omega _{X})=-\omega _{X}\ {\rm and}\ \widetilde{\upsilon _{+,even}}^{\ast }(\omega _{X})=\omega _{X}. \]
Thus for $\{\varsigma ,\varsigma '\}$, one acts on $H^{0}(X,\Omega ^{2n}_{X})$ as ${\rm id}$, and another act on $H^{0}(X,\Omega ^{2n}_{X})$ as $-{\rm id}$.
\end{proof}
We put $\varsigma _{+}\in \{\varsigma ,\varsigma '\}$ as acts on $H^{0}(X,\Omega ^{2n}_{X})$ as ${\rm id}$ and $\varsigma _{-}\in \{\varsigma ,\varsigma '\}$ as acts on $H^{0}(X,\Omega ^{2n}_{X})$ as $-{\rm id}$.
\begin{pro}\label{dfn:50}
Suppose $E$ has numerically trivial involutions.
Let $E^{[n]}$ be the Hilbert scheme of $n$ points of $E$, $\pi :X\rightarrow E^{[n]}$ the universal covering space of $E^{[n]}$, $\rho $ the covering involution of $\pi $, and $n\geq 3$. 
Let $\iota $ be an involution of $X$ which $\iota ^{\ast }$ acts on $H^{2}(X,{\mathbb C})$ as ${\rm id}$ and on $H^{0}(X,\Omega ^{2n}_{X})$ as $-{\rm id}$, and $\iota \not=\rho $.
Then we have $\iota =\varsigma _{-}$.
\end{pro}
\begin{proof}
Let $\iota $ be an involution of $X$ which acts on $H^{2}(X,{\mathbb C})$ as ${\rm id}$ and on $H^{0}(X,\Omega ^{2n}_{X})$ as $-{\rm id}$, and $\iota \not=\rho $.
By Remark $\ref{dfn:1002}$, 
$\iota |_{X\setminus \pi ^{-1}(D_{E})}$ is an automorphism of $X\setminus \pi ^{-1}(D_{E})$.
By the uniqueness of the universal covering space, there is an automorphism $g$ of $K^{n}\backslash \Gamma _{K}$ 
such that $\iota \circ \omega =\omega \circ g$:
$$
\xymatrix{
K^{n}\setminus \Gamma _{K} \ar[d]^{\omega } \ar[r]^{g} &K^{n}\setminus \Gamma _{K} \ar[d]^{\omega } \\
X\setminus \pi ^{-1}(D_{E}) \ar[r]^{\iota }&X\setminus \pi ^{-1}(D_{E}).
}
$$
Like the proof of Proposition $\ref{dfn:12}$, we can assume that there are some automorphisms $g_{i}$ of $K$ such that $g=g_{1}\times \cdots \times g_{n}$,  for each $1\leq i\leq n$, $g_{i}=g_{1}$ or $g_{i}=g_{1}\circ \sigma $, and $g_{1}\circ \sigma =\sigma \circ g_{1}$.
Since $\iota ^{2}={\rm id}_{X}$, so we have $g^{2}\in H$. Thus we have $g^{2}={\rm id}_{K^{n}}$ or $\sigma _{i_{1}\ldots i_{k}}$. By $[\ref{bio:9},\,{\rm Lemma}\,1.2]$,  
we have $g^{2}={\rm id}_{K^{n}}$.
We put $g':=g_{1}$.
Let $g'_{E}$ be the induced automorphism of $E$ by $g'$, and $g'^{[n]}_{E}$ the induced automorphism of $E^{[n]}$ by $g'_{E}$.
Since $g'^{[n]}_{E}\circ \pi =\pi \circ \iota $ and $n\geq 3$,
$g'^{[n]\ast }_{E}$ acts on $H^{2}(E^{[n]},{\mathbb C})$ as ${\rm id}$, and therefore $g'^{\ast }_{E}$ acts on $H^{2}(E,{\mathbb C})$ as $id$.
If $g'_{E}={\rm id}_{E}$, then we have $\iota =\rho $ or ${\rm id}_{X}$, a contradiction.
Since $g^{2}={\rm id}_{K^{n}}$
Thus the order of $g'_{E}$ is $2$. Since $g'^{\ast }_{E}$ acts on $H^{2}(E,{\mathbb C})$ as ${\rm id}$,
we have $g'_{E}=\upsilon $, and therefore $g'=\upsilon _{+}$ or $g'=\upsilon _{-}$. 
By the definition of $\varsigma $ and $\varsigma '$, we obtain $\iota =\varsigma $ or $\iota =\varsigma '$.
Since $\iota ^{\ast }$ acts on $H^{0}(X,\Omega ^{2n}_{X})$ as $-{\rm id}$, we obtain $\iota =\varsigma _{-}$.
\end{proof}
\begin{thm}\label{dfn:131}
Let $E$ be an Enriques surface, $E^{[n]}$ the Hilbert scheme of $n$ points of $E$, $\pi :X\rightarrow E^{[n]}$ the universal covering space of $E^{[n]}$, and $n\geq 3$.
If $X$ has a involution $\iota $ which $\iota ^{\ast }$ acts on $H^{2}(X,{\mathbb C})$ as ${\rm id}$, and $\iota \not=\rho $.
Then $E$ has a numerically trivial involution.
\end{thm}
\begin{proof}
Let $\iota $ be an involution of $X$ which acts on $H^{2}(X,{\mathbb C})$ as ${\rm id}$, and $\iota \not=\rho $.
By {\rm Remark} $\ref{dfn:1002}$, 
$\iota |_{X\setminus \pi ^{-1}(D_{E})}$ is an automorphism of $X\setminus \pi ^{-1}(D_{E})$.
By the uniqueness of the universal covering space, there is an automorphism $g$ of $K^{n}\backslash \Gamma _{K}$ 
such that $\iota \circ \omega =\omega \circ g$:
$$
\xymatrix{
K^{n}\setminus \Gamma _{K} \ar[d]^{\omega } \ar[r]^{g} &K^{n}\setminus \Gamma _{K} \ar[d]^{\omega } \\
X\setminus \pi ^{-1}(D_{E}) \ar[r]^{\iota }&X\setminus \pi ^{-1}(D_{E}).
}
$$
Like the proof of Proposition $\ref{dfn:12}$, we can assume that there are some automorphisms $g_{i}$ of $K$ such that $g=g_{1}\times \cdots \times g_{n}$,  for each $1\leq i\leq n$, $g_{i}=g_{1}$ or $g_{i}=g_{1}\circ \sigma $, and $g_{1}\circ \sigma =\sigma \circ g_{1}$.
Since $\iota ^{2}={\rm id}_{X}$, we have $g^{2}\in H$. Thus we have $g^{2}={\rm id}_{K^{n}}$ or $\sigma _{i_{1}\ldots i_{k}}$. By $[\ref{bio:9},\,{\rm Lemma}\,1.2]$,  
we have $g^{2}={\rm id}_{K^{n}}$.
We put $g':=g_{1}$.
Let $g'_{E}$ be the induced automorphism of $E$ by $g'$, and $g'^{[n]}_{E}$ the induced automorphism of $E^{[n]}$ by $g'_{E}$.
Since $g'^{[n]}_{E}\circ \pi =\pi \circ \iota $ and $n\geq 3$,
$g'^{[n]\ast }_{E}$ acts on $H^{2}(E^{[n]},{\mathbb C})$ as ${\rm id}$, and therefore $g'^{\ast }_{E}$ acts on $H^{2}(E,{\mathbb C})$ as ${\rm id}$.
If $g'_{E}={\rm id}$, like the proof of Proposition $\ref{dfn:130}$ we have $\iota =\rho $ or $\iota ={\rm id}_{X}$, a contradiction.
Thus we have $g'_{E}\not={\rm id}$. Since $g^{2}={\rm id}_{K^{n}}$, $g'_{E}$ is an involution of $E$. 
Since $g'^{\ast }_{E}$ acts on $H^{2}(E,{\mathbb C})$ as ${\rm id}$,
$E$ has a numerically trivial involution.
\end{proof}
\begin{lem}\label{lem:1}
dim$_{\mathbb C}H^{2n-1,1}(K^{n}/H,{\mathbb C})=10.$
\end{lem}
\begin{proof}
Let $\sigma $ be the covering involution of $\mu :K\rightarrow E$. Put
\[ H^{k}_{\pm }(K,{\mathbb C})^{p,q}:=\{ \alpha \in H^{k}(K,{\mathbb C})^{p,q}:\sigma ^{\ast }(\alpha )=\pm \alpha  \}\,{\rm and}  \] 
\[ h^{p,q}_{\pm }(K):={\rm dim}_{\mathbb C}H^{k}_{\pm }(K,{\mathbb C})^{p,q}.  \]
Since $K$ is a $K3$ surface, we have 
\[ h^{0,0}(K)=1,\ h^{1,0}(K)=0,\ h^{2,0}(K)=1,\ {\rm and}\ h^{1,1}(K)=20,\,{\rm and} \]
\[ h_{+}^{0,0}(K)=1,\,h_{+}^{1,0}(K)=0,\,h_{+}^{2,0}(K)=0,\,{\rm and}\,h_{+}^{1,1}(K)=10,\,{\rm and} \]
\[ h^{0,0}_{-}(K)=0,\,h^{1,0}_{-}(K)=0,\,h^{2,0}_{-}(K)=1,\,{\rm and}\,h^{2,0}_{-}(K)=10. \]
Let $\Lambda $ be a subset of ${\mathbb Z}_{\geq 0}^{2n}$
\[ \Lambda :=\{ (s_{1},\cdots ,s_{n},t_{1},\cdots ,t_{n})\in {\mathbb Z}_{\geq 0}^{2n}:\Sigma _{i=1}^{n}s_{i}=2n-1,\,\Sigma _{j=1}^{n}t_{j}=1 \}. \] 
From the K\"unneth Theorem, we have
$$
H^{2n}(k^{n},{\mathbb C})^{2n-1,1}\simeq \bigoplus _{(s_{1},\cdots ,s_{n},t_{1},\cdots ,t_{n})\in \Lambda }\biggl{(}\bigotimes _{i=1}^{n}H^{2}(K,{\mathbb C})^{s_{i},t_{i}}\biggl{)}.
$$ 
We fix a basis $\alpha $ of $H^{2}(K,{\mathbb C})^{2,0}$ and a basis $\{\beta _{i}\}_{i=1}^{10}$ of $H^{2}_{-}(K,{\mathbb C})^{1,1}$, and let
\[ \tilde{\beta _{i}}:=\bigotimes _{j=1}^{n}\epsilon _{j}  \]
where $\epsilon _{j}=\alpha $ for $j\not=i$ and $\epsilon _{j}=\beta _{i}$ for $j=i$, and 
\[ \gamma _{i}:=\bigoplus _{j=1}^{n}\tilde{\beta _{j}}. \] 
then we have
\begin{equation}
\label{eq:10}
H^{2n}(K^{n}/H,{\mathbb C})^{2n-1,1}\simeq \bigoplus _{i=1}^{10}{\mathbb C}\gamma _{i},
\end{equation}
dim$_{\mathbb C}H^{2n-1,1}(K^{n}/H,{\mathbb C})=10$.
\end{proof}
Since $X$ and $K^{n}/H$ are projective, $K^{n}/H$ is a V-manifold, and $\pi $ is a surjective, $\pi ^{\ast }:H^{p,q}(K^{n}/H,{\mathbb C})\rightarrow H^{p,q}(X,{\mathbb C})$ is injective.
\begin{thm}\label{thm:91}
We suppose $n\geq 2$. Let $\pi :X\rightarrow E^{[n]}$ be the universal covering space.
For any automorphism $f$ of $X$, if $f^{\ast }$ is acts on $H^{\ast }(X,{\mathbb C}):=\bigoplus _{i=0}^{2n}H^{i}(X,{\mathbb C})$ as identity, then $f={\rm id}_{X}$.
\end{thm}
\begin{proof}
Since $f^{\ast }$ acts on $H^{2}(X,{\mathbb C})$ as identity, $f$ is an automorphism of $K^{n}\setminus \Gamma _{K}/H$. 
Let $p_{H}:K^{2}\setminus \Gamma _{K}\rightarrow K^{2}\setminus \Gamma _{K}/H$ be the natural map.
Then the uniqueness of the universal covering space, we can that there are some automorphisms $g_{i}$ of $K$ such that $g:=g_{1}\times \cdots \times g_{n}$, $g_{i}=g_{1}$ or $g_{i}=g_{1}\circ \sigma $, $g_{1}\circ \sigma =\sigma \circ g_{1}$ for $2\leq i\leq n$, and $f\circ p_{H} =p_{H}\circ g$:
$$
\xymatrix{
K^{n}\setminus \Gamma _{K}/H \ar[r]^{f}&K^{n}\setminus \Gamma _{K}/H \\
K^{n}\setminus \Gamma _{K} \ar[u]^{p_{H}} \ar[r]^{g} &K^{n}\setminus \Gamma _{K} \ar[u]^{p_{H}}. 
}
$$
Let $g_{H}$ be the induced automorphism of $K^{n}/H$.
Then we obtain $g_{H}\circ \varphi _{X}=\varphi _{X}\circ f$:
$$
\xymatrix{
K^{n}/H \ar[r]^{g_{H}}&K^{n}/H \\
X \ar[u]^{\varphi _{X}} \ar[r]^{f} &X \ar[u]^{\varphi _{X}}.
}
$$
Put $g_{1E}$ the automorphism of $E$ induced by $g_{1}$. 
Since $f^{\ast }$ acts on $H^{2}(X,{\mathbb C})$ as identity, $g_{H}^{\ast }$ acts on $H^{2}(K^{n}/H,{\mathbb C})$ as identity. 
Since $H^{2}(K^{n}/H,{\mathbb C})\cong H^{2}(E,{\mathbb C})$, $g_{1E}^{\ast }$ acts on $H^{2}(E,{\mathbb C})$ as identity. 
From Lemma $\ref{lem:1}$, we have
\[ H^{2n}(X,{\mathbb C})^{2n-1,1}=\bigoplus _{i=1}^{10}{\mathbb C}\varphi ^{\ast }_{X}\gamma _{i}. \]
Suppose $g_{1}\not=\sigma $ and $g_{1}\not={\rm id}_{K}$. Since $g_{1E}^{\ast }$ acts on $H^{2}(E,{\mathbb C})$ as identity,
from $[\ref{bio:9},\,{\rm page}\,386$-$389]$, the order of $g_{1E}$ is at most $4$.
If the order of $g_{1E}$ is $2$, there is an element $\alpha _{\pm }\in H^{2}_{-}(K,{\mathbb C})^{1,1}$ such that $g_{1}^{\ast }(\alpha _{\pm })=\pm \alpha $.
By the equation $(\ref{eq:10})$ and the proof of Lemma $\ref{dfn:100}$, $f$ does not act on $H^{2n}(X,{\mathbb C})^{2n-1,1}$ as identity, it is a contradiction.
If the order of $g_{1E}$ is $4$, then 
there is an element $\alpha '_{\pm }\in H^{2}_{-}(K,{\mathbb C})^{1,1}$ such that $g_{1}^{\ast }(\alpha '_{\pm })=\pm \sqrt{-1}\alpha '_{\pm }$ from $[\ref{bio:9},\,{\rm page}\,390$-$391]$.
By the equation $(\ref{eq:10})$ and and the proof of Lemma $\ref{dfn:100}$, $f$ does not act on $H^{2n}(X,{\mathbb C})^{2n-1,1}$ as identity, it is a contradiction.
Thus we have $g_{1E}={\rm id}_{E}$, i.e. $g_{1}=\sigma $ or $g_{1}={\rm id}_{K}$, and $f={\rm id}_{X}$ or $f=\rho $ where $\rho $ is the covering involution of $\pi :X\rightarrow E^{n}$.
From Proposition $\ref{dfn:9}$ $H^{2n}(E^{[n]},{\mathbb C})^{2n-1,1}\simeq 0$, $\rho $ does not act on $H^{2n}(X,{\mathbb C})^{2n-1,1}$ as identity.
Since $f^{\ast }$ acts on $H^{2n}(X,{\mathbb C})^{2n-1,1}$ as identity, we have $f={\rm id}_{X}$.
\end{proof}
\begin{cro}\label{cro:92}
We suppose $n\geq 2$. Let $\pi :X\rightarrow E^{[2]}$ be the universal covering space.
For any two automorphisms $f$ and $g$ of $X$, if $f^{\ast }=g^{\ast }$ on $H^{\ast }(X,{\mathbb C})$, then $f=g$.
\end{cro}

By $[\ref{bio:9},\,{\rm Proposition}\,1.1]$, 
there is just one automorphism of $E$, denoted $\upsilon $, such that its order is $2$, and $\upsilon ^{\ast }$ acts on $H^{2}(E,{\mathbb C})$ as ${\rm id}$.
For $\upsilon $, there are just two involutions of $K$ which are liftings of $\upsilon $,
one acts on $H^{0}(K,\Omega ^{2}_{K})$ as ${\rm id}$, 
and another acts on $H^{0}(K,\Omega ^{2}_{K})$ as $-{\rm id}$, we denote by $\upsilon _{+}$ and $\upsilon _{-}$, respectively. 
Then they satisfies $\upsilon _{+}=\upsilon _{-}\circ \sigma $.
Let $\upsilon ^{[n]}$ be the automorphism of $E^{[n]}$ which is induced by $\upsilon $.
For $\upsilon ^{[n]}$, there are just two automorphisms of $X$ which are liftings of $\upsilon ^{[n]}$, denoted $\varsigma $ and $\varsigma '$, respectively.
Then they satisfies $\varsigma =\varsigma '\circ \sigma $, and each order of $\varsigma $ and $\varsigma '$ is $2$.
From ${\rm Lemma}5.11$, one acts on $H^{0}(X,\Omega ^{2n}_{X})$ as ${\rm id}$, and another act on $H^{0}(X,\Omega ^{2n}_{X})$ as $-{\rm id}$.
We put $\varsigma _{+}\in \{\varsigma ,\varsigma '\}$ as acts on $H^{0}(X,\Omega ^{2n}_{X})$ as ${\rm id}$ and $\varsigma _{-}\in \{\varsigma ,\varsigma '\}$ as acts on $H^{0}(X,\Omega ^{2n}_{X})$ as $-{\rm id}$.
\begin{thm}\label{thm:3}
Let $E$ and $E'$ be two Enriques surfaces, $E^{[n]}$ and $E'^{[n]}$ the Hilbert scheme of $n$ points of $E$ and $E'$, $X$ and $X'$ the universal covering space of $E^{[n]}$ and $E'^{[n]}$, and $n\geq 3$.
If $X\cong X'$, then $E^{[n]}\cong E'^{[n]}$, i.e. when we fix $X$, then there is just one isomorphism class of the Hilbert schemes of $n$ points of Enriques surfaces such that they have it as the universal covering space.
\end{thm}
\begin{proof}
For an involution of $X$ which is the covering involution of some the Hilbert scheme of $n$ points of Enriques surfaces acts on $H^{2}(X,{\mathbb C})$ as ${\rm id}$, $H^{0}(X,\Omega ^{2n}_{X})$ as $-{\rm id}$, 
and $H^{2n}(X,{\mathbb C})^{2n-1,1}$ as $-{\rm id}$.
From ${\rm Proposition}5.12$, the automorphisms which acts on $H^{2}(X,{\mathbb C})$ as ${\rm id}$, $H^{0}(X,\Omega ^{2n}_{X})$ as $-{\rm id}$, are only $\rho $ and $\varsigma _{-}$.  From the definition of $\varsigma _{-}$ and Lemma $\ref{lem:1}$, $\varsigma _{-}$ does not act on $H^{2n}(X,{\mathbb C})^{2n-1,1}$ as $-{\rm id}$.
Thus we have an argument.
\end{proof}
\section*{Acknowledgements}
I would like to express my thanks to Professor Keiji Oguiso for his advice and encouragement as my supervisor.

\end{document}